\documentclass[11pt]{articlefederico}

\usepackage{enumerate, amsmath, amsfonts, amssymb, amsthm, wasysym, graphics, graphicx, xcolor, url, hyperref, hypcap, a4wide, pdflscape, multido, overpic, caption, bm}
\usepackage[all]{xy}
\usepackage{tikz}\usetikzlibrary{trees,snakes,shapes,arrows,matrix,calc}
\graphicspath{{figures/}}
\usepackage{overpic} 

\newtheorem{theorem}{Theorem}[section]
\newtheorem{corollary}[theorem]{Corollary}
\newtheorem{proposition}[theorem]{Proposition}
\newtheorem{lemma}[theorem]{Lemma}
\newtheorem{definition}[theorem]{Definition}

\newtheorem{problem}[theorem]{Problem}

\theoremstyle{definition}

\newtheorem{remark}[theorem]{Remark}

\makeatletter
\newtheorem*{rep@theorem}{\rep@title}
\newcommand{\newreptheorem}[2]{%
\newenvironment{rep#1}[1]{%
 \def\rep@title{#2 \ref{##1}}%
 \begin{rep@theorem}}%
 {\end{rep@theorem}}}
 \makeatother
 
\newreptheorem{theorem}{Theorem}
\newreptheorem{corollary}{Corollary}
\newreptheorem{conjecture}{Conjecture}
\renewcommand{\S}{\mathcal{S}} 

\newcommand{\shape}{\operatorname{shape}} 
\newcommand{\sh}{\operatorname{sh}} 
\newcommand{\diam}{\operatorname{diam}} 

\newcommand{\Left}{L_{m,n}} 
\newcommand{\Leftt}{L_{m,n}^+} 
\newcommand{\Hor}{H_{m,n}}  
\newcommand{\LeftLambda}{L_{\lambda}} 
\newcommand{\LefttLambda}{L_{\lambda}^+} 
\newcommand{\HLambda}{H_{\lambda}}  
\newcommand{\LeftLambdaP}[1]{L_{#1}} 
\newcommand{\LefttLambdaP}[1]{L_{#1}^+} 
\newcommand{\HLambdaP}[1]{H_{#1}}  

\newcommand{\LeftLambdak}{L_{\lambda,k}} 
\newcommand{\LefttLambdak}{L_{\lambda,k}^+} 

\usepackage{todonotes}

\begin{document}
\title{\textsf{The configuration space of a robotic arm in a tunnel.}}
\author{
\textsf{Federico Ardila\footnote{\noindent \textsf{San Francisco State University, San Francisco, USA; Univ. de Los Andes, Bogot\'a, Colombia; federico@sfsu.edu.}}} \\
\and 
\textsf{Hanner Bastidas\footnote{\noindent \textsf{Departamento de Matem\'aticas, Universidad del Valle, Cali, Colombia; hanner.bastidas@correounivalle.edu.co.}}}
\and
\textsf{Cesar Ceballos\footnote{\noindent \textsf{Faculty of Mathematics, University of Vienna, Austria; cesar.ceballos@univie.ac.at.}}}
\and
\textsf{John Guo\footnote{\noindent \textsf{San Francisco State University, San Francisco, USA; jguo@mail.sfsu.edu.
\newline 
This project is part of the SFSU-Colombia Combinatorics Initiative. 
FA was  supported by the US NSF CAREER Award DMS-0956178 and NSF Combinatorics Award DMS-1600609. CC was  supported by 
a Banting Postdoctoral Fellowship
of the Canadian government, a York University research grant, 
and the Austrian FWF Grant F 5008-N15. 
}}}}
\date{}
\maketitle

\begin{abstract} 
We study the motion of a robotic arm inside a rectangular tunnel. We prove that the configuration space  of all possible positions of the robot is a $\mathrm{CAT}(0)$ cubical complex. This allows us to use techniques from geometric group theory to find the optimal way of moving the arm  from one position to another. We also compute the diameter of the configuration space, that is, the longest distance between two positions of the robot.

\smallskip
\noindent
\end{abstract}

\vspace{-0.6cm}


\section{\textsf{Introduction}}
We consider a robotic arm $R_{m,n}$ of length $n$ moving in a rectangular tunnel of width $m$ without self-intersecting. The robot consists of $n$ links of unit length, attached sequentially, and facing up, down, or right. The base is affixed to the lower left corner. Figure~\ref{fig:armingrid} illustrates a possible position of the arm of length 8 in a tunnel of width $2$.


\begin{figure}[htbp]
	\centering
		\includegraphics[height = 1.2cm]{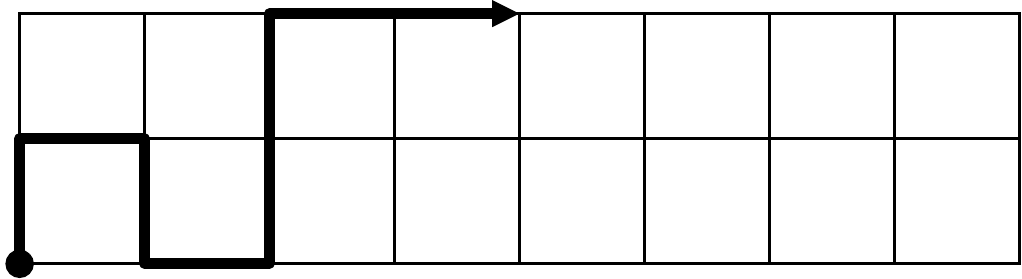}
	\caption{The robotic arm $R_{8,2}$.}
	\label{fig:armingrid}
\end{figure}

\noindent
The robotic arm may move freely using two kinds of local moves:

\noindent $\bullet$ \emph{Switch corners:} Two consecutive links facing different directions interchange directions.

\noindent $\bullet$ \emph{Flip the end:} The last link of the robot rotates $90^\circ$ without intersecting itself.

\begin{figure}[htbp]
	\centering
    \includegraphics[height = 1.5cm]{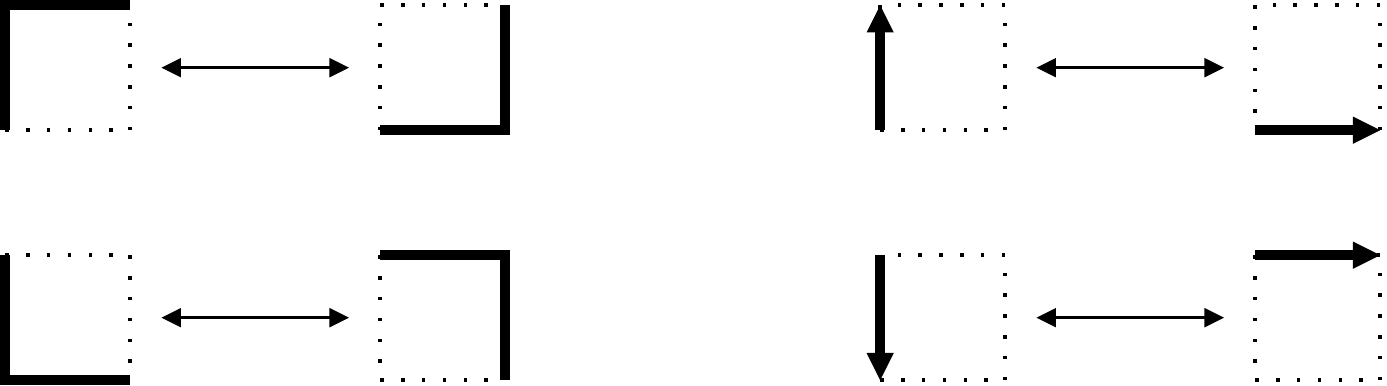}
	\caption{The two kinds of local moves of the robotic arm.}
	\label{fig:mov}
\end{figure}	

\noindent We study the following fundamental problem. 
	
\begin{problem}\label{prob:moverobot}
Move the robotic arm $R_{m,n}$ from one given position to another optimally.
\end{problem}	

When we are in a city we do not know well and we are trying to get from one location to another, we will usually consult a map of the city to plan our route. This is a simple but powerful idea. Our strategy to approach Problem \ref{prob:moverobot} is pervasive in many branches of mathematics:
we will be to build and understand the ``map" of all possible positions of the robot; we call it the \emph{configuration space} $\S_{m,n}$. Such spaces are also called \emph{state complexes}, \emph{moduli spaces}, or \emph{parameter spaces} in other fields. 
Following work of  Abrams--Ghrist~\cite{AG} in applied topology and Reeves~\cite{Re} in geometric group theory, Ardila, Owen, and Sullivant~\cite{AOS} and Ardila, Baker, and Yatchak~\cite{ABY} showed how to solve Problem \ref{prob:moverobot} for robots whose configuration space is \emph{CAT(0)}; this is a notion of non-positive curvature defined in Section \ref{sec:CAT(0)}. 
This is the motivation for our main result.

\begin{theorem}\label{th:CAT(0)}
The configuration space $\S_{m,n}$ of the robotic arm $R_{m,n}$ of length $n$ in a tunnel of width $w$ is a $\mathrm{CAT}(0)$ cubical complex.
\end{theorem}

\begin{figure}[htbp]
	\centering
	\includegraphics[height=4cm]{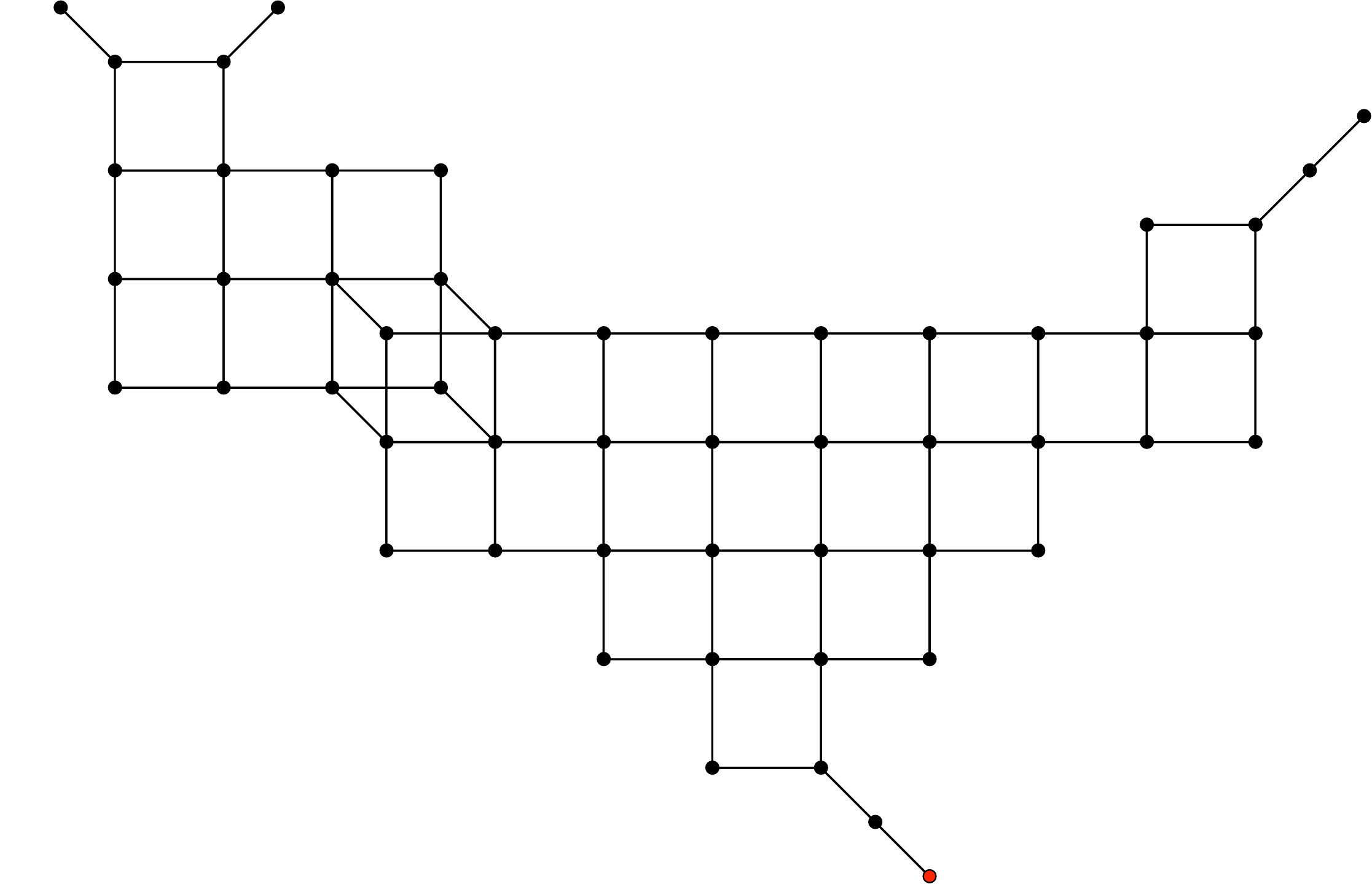}
	\caption{The configuration space $\S_{6,2}$ of the robotic arm $R_{6,2}$.}
	\label{fig:gft6}
\end{figure}

In light of \cite{ABY}, Theorem \ref{th:CAT(0)} provides a solution to Problem \ref{prob:moverobot}.

\begin{corollary}
There is an explicit algorithm, implemented in Python, to move the robotic arm $R_{m,n}$ optimally from one given position to another.
\end{corollary}

The reader may visit 
\href{http://math.sfsu.edu/federico/Articles/movingrobots.html}{\texttt{http://math.sfsu.edu/federico/Articles/movingrobots.html}} 
to watch a video preview of this program or to download the Python code.


Abrams and Ghrist showed that cubical complexes appear in a wide variety of contexts where a discrete system moves according to local rules -- such as domino tilings, reduced words in a Coxeter group, and non-intersecting particles moving in a graph -- and that these cubical complexes are sometimes CAT(0)~\cite{AG}.
The goal of this paper is to illustrate how the techniques of \cite{ABY, AOS} may be used to prove that configuration spaces are CAT(0), uncover their underlying combinatorial structure, and navigate them optimally.

To our knowledge, Theorem \ref{th:CAT(0)} provides one of the first infinite families of configuration spaces which are CAT(0) cubical complexes, and the first one to 
invoke the full machinery of \emph{posets with inconsistent pairs}  of \cite{ABY, AOS}. It is our hope that the techniques employed here will provide a blueprint for the study of CAT(0) cubical complexes in many other contexts.


We now describe the structure of the paper. In Section \ref{sec:prelim} we define more precisely the configuration space $\S_{m,n}$ of the robotic arm $R_{m,n}$. 
Section \ref{sec:Euler} is devoted to collecting some preliminary enumerative evidence for our main result,  Theorem \ref{th:CAT(0)}.
It follows from very general results of Abrams and Ghrist \cite{AG} that the configuration space $\S_{m,n}$ is a cubical complex. Also, we know from work of Gromov \cite{Gr} that $\S_{m,n}$ will be  $\mathrm{CAT}(0)$  if and only if it is contractible. Before launching into the proof  that $\S_{m,n}$ is CAT(0), we first verify in the special case $m=2$ that this space has the correct Euler characteristic, equal to $1$. We do it as follows.

\begin{theorem}\label{th:faceenum}
If $c_{n,d}$ denotes the number of $d$-dimensional cubes in the configuration space $\S_{2,n}$ for a robot of width $2$, then 
\[
 \sum_{n, d \geq 0} c_{n,d} \, x^ny^d  = 
  \frac{1+x^2+2x^3-x^4+xy+x^2y+4x^3y+x^4y+x^3y^2+2x^4y^2+x^5y^2}
{1-2x+x^2-x^3-x^4-2x^4y-2x^5y-x^5y^2-x^6y^2}.
\]
\end{theorem}

The Euler characteristic of $\S_{2,n}$ is $\chi(\S_{2,n}) = c_{n,0} - c_{n,1} + \cdots$.  Substituting $y=-1$ above we find, in an expected but satisfying miracle of cancellation, that the generating function for $\chi(\S_{2,n})$ is $1/(1-x) = 1+x+x^2+\cdots$.We obtain:

\begin{theorem}\label{th:Euler}
The Euler characteristic of the configuration space $\S_{2,n}$ equals $1$.
\end{theorem}

In Section~\ref{sec:CAT(0)} we collect the tools we use to prove Theorem \ref{th:CAT(0)}. Ardila, Owen, and Sullivant~\cite{AOS} gave a bijection between $\mathrm{CAT}(0)$ cubical complexes~$X$ and combinatorial objects $P(X)$ called \emph{posets with inconsistent pairs} or {PIP}s. 
 The PIP $P(X)$ is usually much simpler, and serves as a ``remote control" to navigate the space $X$. This bijection allows us to prove that a (rooted) cubical complex is $\mathrm{CAT}(0)$ by identifying its corresponding PIP.\footnote{It is worth noting that PIPs are known as \emph{event systems} \cite{W} in computer science and are closely related to \emph{pocsets} \cite{Ro, Sa} in geometric group theory.}

We use this technique to prove Theorem \ref{th:CAT(0)} in Section \ref{sec:ourPIP}. To do this, we introduce the \emph{coral PIP} $C_{m,n}$ of \emph{coral tableaux}, illustrated in Figure \ref{fig:coralPIP}. By studying the combinatorics of coral tableaux, we show the following:

\begin{proposition}\label{conj:PIP}
The \emph{coral PIP} $C_{m,n}$ is the PIP that certifies that the configuration space $\S_{m,n}$ is a CAT(0) cubical complex.
\end{proposition}

\begin{figure}[h]
	\centering
	\includegraphics[width = 0.655 \textwidth]{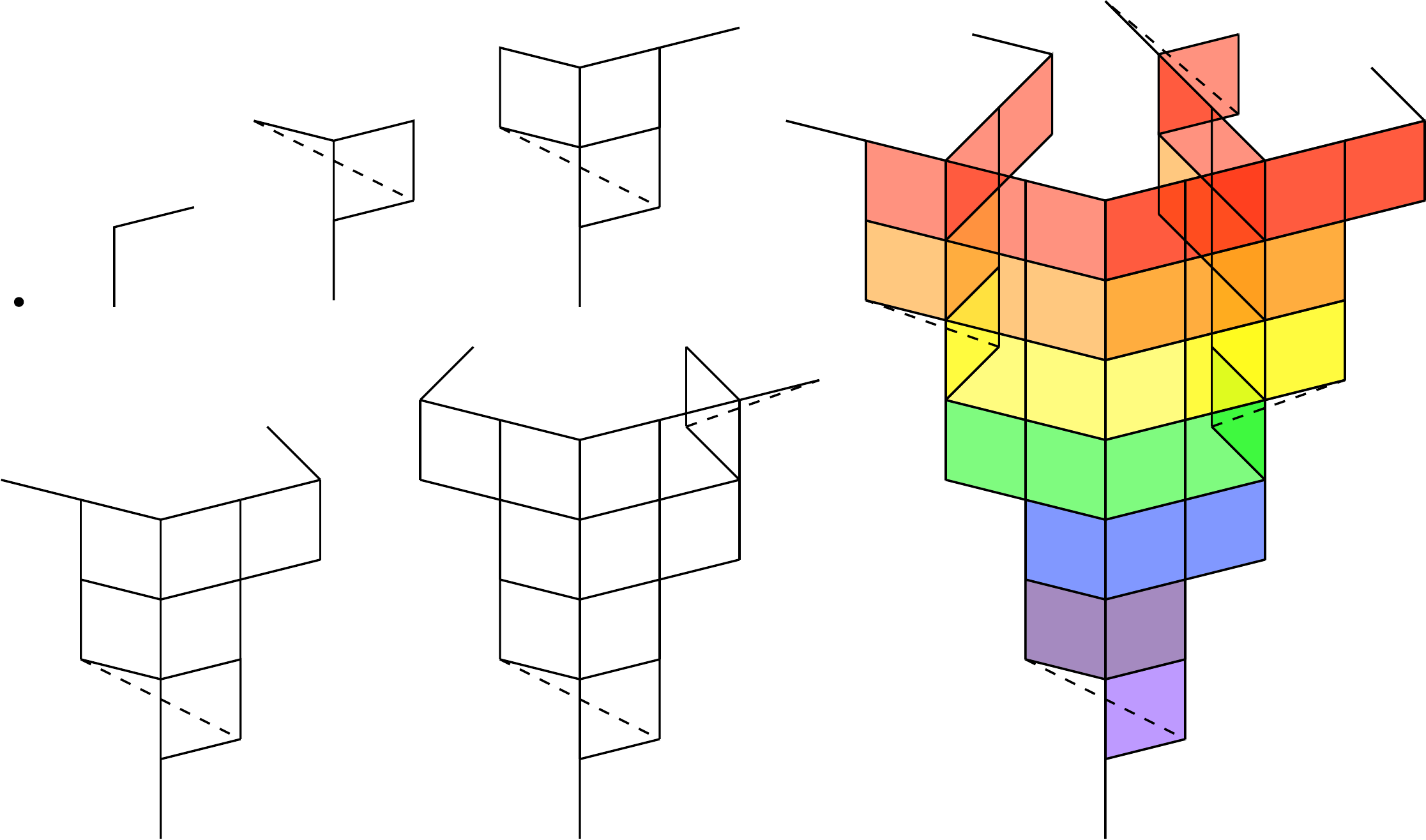}
	\caption{The coral PIPs (``remote controls") for the robotic arms of length $1,2,3,4,5,6,9$ in a tunnel of width $2$.}
	\label{fig:coralPIP}
\end{figure}

In Section \ref{sec:diameter} we use the combinatorics of coral tableaux to find a combinatorial formula for the distance between any two positions of the robotic arm. This allows us to find the longest of such distances --  that is, the diameter of the configuration space $\S_{m,n}$ in the $l_1$-metric -- answering a question of the first author. \cite{ArdilaECCO}

\begin{theorem}
The diameter of the transition graph of the robot $R_{m,n}$ of length $n$ in a tunnel of width $m\geq 2$ is 
\[
\diam G(R_{m,n}) =
\left\{ \begin{array}{lcl}
        d(\Left,\Hor) & \mbox{for} & n<6 \\ 
	d(\Left,\Leftt)  & \mbox{for} & n\geq 6,
                \end{array}\right.
\]
where $\Left=u^mrd^mru^mr\dots$ is the \emph{left justified} position, $\Leftt=(urdr)L_{m,n-4}$ and $\Hor$ is the \emph{fully horizontal} position (see Figure~\ref{fig:robot_left}). A precise formula is stated in Theorem~\ref{thm:diameter}.
\end{theorem}

Finally, in Section \ref{sec:algorithm} we discuss our solution to Problem \ref{prob:moverobot}. 
As explained in \cite{ABY}, we use the coral PIP as a remote control for our robotic arm. This allows us to implement an algorithm to move the robotic arm $R_{m,n}$ optimally from one position to another.

\section{\textsf{The configuration space}}\label{sec:prelim}
 
Recall that $R_{m,n}$ is a robotic arm consisting of $n$ links of unit length attached sequentally, and facing up, down, or right. The robot is inside a rectangular tunnel of width $m$ and pinned down at the bottom left corner of the tunnel. It moves inside the tunnel without self-intersecting, by switching corners and flipping the end as illustrated in Figure \ref{fig:mov}. 
 
In this section we describe the configuration space $\mathcal{S}_{m,n}$ of all possible positions of the robot $R_{m,n}$. 
We begin by considering the \emph{transition graph} $G(R_{m,n})$ whose vertices are the possible states of the robot, and whose edges correspond to the allowable moves between them.
Figure~\ref{fig:grafodetrancisionn=4} illustrates the transition graph $G(R_{2,4})$
 of a robotic arm of length $4$.

\begin{figure}[htbp]
	\centering
		\includegraphics[width = \textwidth]{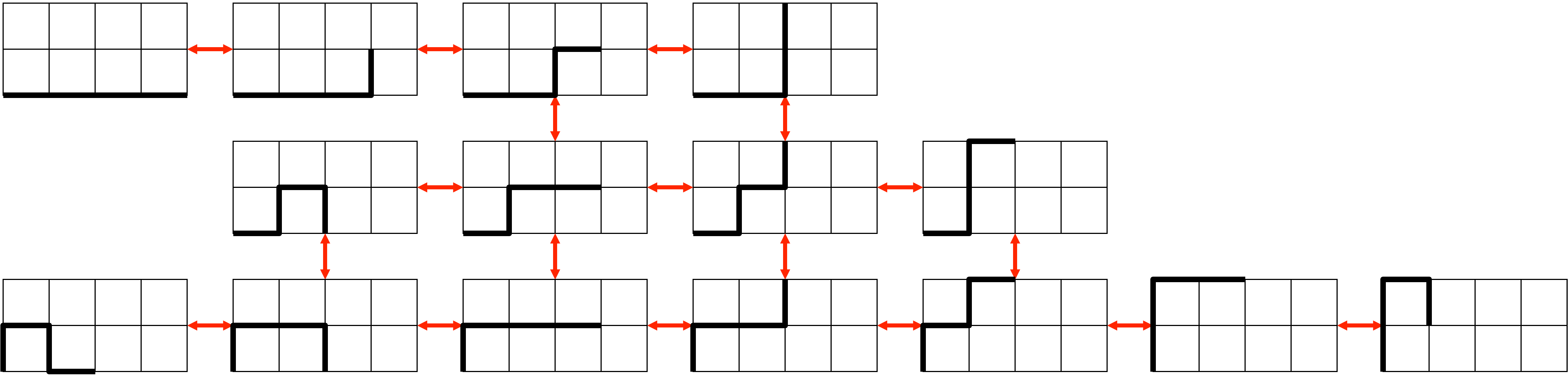}
	\caption{The transition graph of a robotic arm of length $4$.}
	\label{fig:grafodetrancisionn=4}
\end{figure}

As these examples illustrate, each one of these graphs is the $1$-skeleton of a cubical complex. For example, consider a position $u$ which has two legal moves $a$ and $b$ occuring in disjoint parts of the arm. We call $a$ and $b$ \emph{physically independent} or  \emph{commutative} because $a(b(u)) = b(a(u))$. In this case, there is a square connecting $u, a(u), b(a(u))=a(b(u)),$ and $b(u)$ in $G(R_{m,n})$. 

More generally, we obtain the configuration space by filling all the cubes of various dimensions that we see in the transition graph. Let us make this precise.

\begin{definition} \label{def:configuration space}The \emph{configuration space} or \emph{state complex} $\mathcal{S}_{m,n}$ of the robot $R_{m,n}$ is  the following cubical complex. The   vertices correspond to the states of $R_{m,n}$. An edge between vertices $u$ and $v$ corresponds to a legal move which takes the robot between positions~$u$ and~$v$. 
The $k$-cubes correspond to $k$-tuples of commutative moves: 
Given $k$ such moves which are applicable at a state $u$, we can obtain $2^k$ different states from $u$ by performing a subset of these $k$ moves; these are the vertices of a $k$-cube in $\mathcal{S}_{m,n}$. We endow $\mathcal{S}_{m,n}$ with a Euclidean metric by letting each $k$-cube be a unit cube.
\end{definition}

\noindent Figure \ref{fig:gft6} shows the configuration space $\S_{2,6}$; every square and cube in the diagram is filled in.

\section{\textsf{Face enumeration and the Euler characteristic of $\S_{2,n}$} \label{sec:Euler}}

The main structural result of this paper, Theorem \ref{th:CAT(0)}, is that the configuration space $\S_{m,n}$ of our robot is a \emph{CAT(0) cubical complex}. This is a subtle metric property defined in Section \ref{sec:CAT(0)} and proved in Section \ref{sec:ourPIP}. As a prelude, this section is devoted to proving a partial result in that direction. 
It is known \cite{BH, Gr} that CAT(0) spaces are contractible, and hence have Euler characteristic equal to $1$. We now prove:

\begin{reptheorem}{th:Euler}
The Euler characteristic of the configuration space $\S_{2,n}$ equals $1$.
\end{reptheorem}

\noindent This provides enumerative evidence for our main result in width $m=2$. While we were stuck for several weeks
trying to prove Theorem \ref{th:CAT(0)}, we found this evidence very encouraging.

To prove Theorem \ref{th:Euler}, our strategy is to compute the $f$-vector of $\S_{2,n}$.
Recall that the \emph{$f$-vector} of a $d$-dimensional polyhedral complex $X$ is $f_X=(f_0,f_1,\dots,f_d)$ where $f_k$ is the number of $k$-dimensional faces. The Euler characteristic of $X$ is 
$\chi (X) =  f_0 - f_1 + \dots + (-1)^d f_d$. Table~\ref{tab:nestarm} shows the $f$-vectors of the cubical complexes of the robotic arms of length $n\leq 6$. For instance, the complex of Figure \ref{fig:gft6} contains $53$ vertices, $81$ edges, $30$ squares, and $1$ cube. We now carry out this computation for all $n$.

\begin{table}[h]
	\centering
		\begin{tabular}{|c|c|c|c|c|c|c|}
		\hline
		$n$ &  $f_0$ & $f_1$ & $f_2$ & $f_3$ & $\chi (\S_{2,n})$  \\ \hline
		1 &  2 & 1 & 0 & 0 & 1 \\ \hline
		2 &  4 & 3 & 0 & 0 & 1 \\ \hline
		3 &  8 & 8 & 1 & 0 & 1 \\  \hline
		4 &  15 & 18 & 4 & 0 & 1 \\ \hline
		5 &  28 & 38 & 11 & 0 & 1 \\ \hline
		6 &  53 & 81 & 30 & 1 & 1 \\ \hline
		\end{tabular}
	\caption{The $f$-vectors of the cubical complexes $\S_{2,n}$ for arms of length $n\leq 6$.}
	\label{tab:nestarm}
\end{table}

%

\vspace{-.5cm}

\subsection{\textsf{Face enumeration}}
We now compute the generating function for the $f$-vectors of the configuration spaces $\S_{2,n}$. We proceed in several steps.

\subsubsection{\textsf{Cubes and partial states}}
Consider a $d$-cube in the configuration space $\S_{2,n}$; it has $2^d$ vertices. If we superimpose the corresponding $2^d$ positions of the robotic arm, we obtain a sequence of edges, squares, and possibly a ``claw" in the last position, as illustrated in Figure \ref{fig:ejemest}. The number of squares (including the claw if it is present) is $d$, corresponding to the $d$ physically independent moves being represented by this cube. We call the resulting diagram a \emph{partial state}, and let its \emph{weight} be $x^ny^d$. The partial states of weight $x^ny^d$ are in bijection with the $d$-cubes of $\S_{2,n}$.

\begin{figure}[h]
 \centering
 \includegraphics[height = 1.7cm]{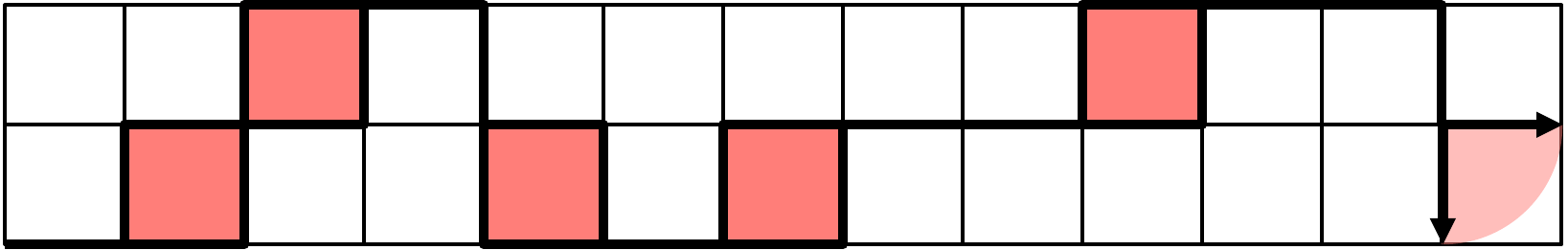}
 \caption{A partial state corresponding to a $6$-cube in the configuration space $\S_{2,20}$.}   \label{fig:ejemest}
\end{figure}

Each partial state gives rise to a word in the alphabet $\{r, v, \ell, \square, \llcorner\}$, where: 

\noindent $\bullet$
$r$ represents a horizontal link of the robot facing to the right. Its weight is $x$.  

\noindent $\bullet$
$v$ represents a vertical link. Its weight is $x$. 

\noindent $\bullet$
$\square$ represents a square, which comes from a move that switches corners of two consecutive links facing different directions. Its weight is $x^2y$. 
  \begin{figure}[h]
  \centering
  \includegraphics[height = 0.9cm]{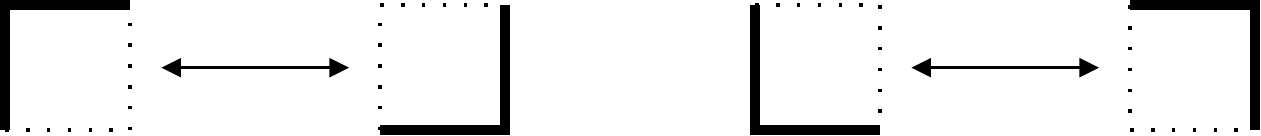}
  \label{fig:box}
  \end{figure}

\noindent $\bullet$
$\llcorner$ represents a claw, which comes from a move that flips the end of the robot, with the horizontal link facing to the right. Its weight is $xy$.
  \begin{figure}[h]
  \centering
  \includegraphics[height = 1.1cm]{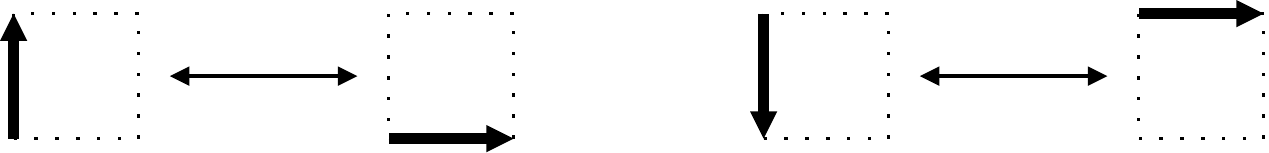}
  \label{fig:l}
  \end{figure}

For example, the partial state of Figure \ref{fig:ejemest} gives rise to the word $w=r \square \square r v \square r \square r r \square r r v \llcorner$.
The weight of the partial state is the product of the weights of the individual symbols; in this case it is $x(x^2y)(x^2y)xx(x^2y)x(x^2y)xx(x^2y)xxx(xy) = x^{20}y^6$.
It is worth remarking that this word does not determine the partial state uniquely; the reader is invited to construct another state giving rise to the same word $w$  above.

\subsubsection{\textsf{Factorization of partial states into irreducibles}}
Our next goal is to use generating functions to enumerate all \emph{partial states} according to their length and dimension. The key idea is that we can ``factor" a partial state uniquely as a concatenation of irreducible factors. Each time the partial state enters one of the borders of the tunnel, we start a new factor. For example, the factorization of the partial state of Figure~\ref{fig:ejemest} is shown in Figure~\ref{fig:factors}. 

\begin{figure}[h]
 \centering
 \includegraphics[height = 1.7cm]{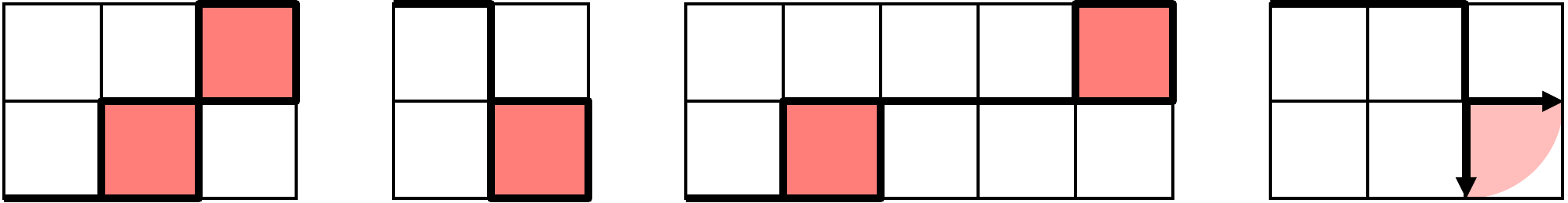}
 \caption{The partial state of Figure \ref{fig:ejemest} has a factorization of the form $M_1 M_5 M_1 F_7$. (See Tables \ref{tab:movfuncgen} and \ref{tab:final}.)}
  \label{fig:factors}
\end{figure}

\begin{definition} Let $P$ be the set of all partial states of robotic arms in a tunnel of width $2$.

\noindent (a) A partial state of the robot is called \emph{irreducible} if 

 $\bullet$ its first step is a horizontal link along the bottom border of the tunnel, and 

 $\bullet$ 
its final step is vertical or square, and is its first arrival at the same or opposite border. 

\noindent(b) A partial state of the robot is called \emph{irreducible final} if it is empty or

 $\bullet$  its first step is a horizontal link along the bottom border of the tunnel, and

 $\bullet$  either its final step is a claw which is also its first arrival at the same or opposite border, 

 or it never arrives at the same or the opposite border.

\noindent 
Let $M$ and $F$ be the sets of irreducible and irreducible final partial states.
\end{definition}

Let $\S=\bigcup_{n=0}^{\infty}\S_{2,n}$ be the disjoint union of the configuration spaces of all robotic arms of all lengths in width 2. Let $B^*$ be the collection of all words that can be made with alphabet~$B$. For instance, $a^* = \{ \emptyset, a, aa, aaa, aaaa,\dots \}$ and  $\{ a, b\}^* = \{ \emptyset, a, b, aa, ab, ba, bb, aaa, aab, \dots \}$.

\begin{proposition}\label{prop:MF}
The partial states in $\S$ starting with a right step $r$ are in weight-preserving bijection with the words in $M^*F$; that is, each partial state in $\S$ corresponds to a unique word of the form $m_1 m_2 \ldots m_\ell f$ 
with $m_i \in M$ and~$f \in F$.
\end{proposition}

\begin{proof}
It is clear from the definitions that every partial state that starts with a horizontal step $r$ factors uniquely as a concatenation $m_{1}^{\pm} m_{2}^{\pm} \ldots m_{\ell}^{\pm} f^{\pm}$ where each $m_i \in M$, $f \in F$, and $p^{\pm}$ equals $p$ or its reflection $p^-$ across the horizontal axis. It remains to observe that whether $m_i^\pm$ is $m_i$ or $m_i^-$, and whether $p^\pm$ is either $p$ or $p^-$, is determined completely by the previous terms of the sequence. 
\end{proof}

\begin{corollary} \label{cor:MF}  If the generating functions for partial states, irreducible partial states, and irreducible final partial states are $C(x,y), M(x,y), F(x,y)$ respectively, then
\[
1+xC(x,y) = \frac{F(x,y)}{1-M(x,y)}.
\]
\end{corollary}

\begin{proof}
This follows from Proposition \ref{prop:MF}. The extra factor of $x$ comes from the fact that Proposition \ref{prop:MF} is counting partial states with an initial right step.
\end{proof}

\subsubsection{\textsf{Enumeration of irreducible partial states}}

Let us compute the generating function $M(x,y)$ for irreducible partial states.

\begin{table}[h]
\begin{center}
  \begin{tabular}{|l|c|c|}
    \hline
    Type  & Illustration & Generating function \\ \hline
$M_1= (rr^*)\square (r^*)\square$
    &
    \begin{minipage}{4cm}
        \centering
        \smallskip
    \includegraphics[height = 1cm]{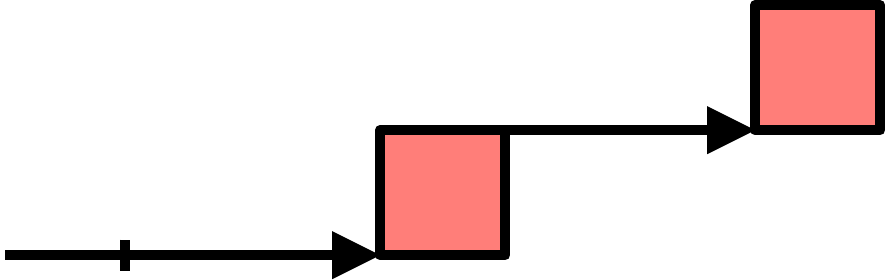}
    \end{minipage}
    & \begin{Large}$\frac{x^5y^2}{(1-x)^2}$ \end{Large} \\ \hline
$M_2= (rr^*)\square (rr^*)\square'$
&
        \begin{minipage}{4cm}
        \centering
        \smallskip
	\includegraphics[height = 1cm]{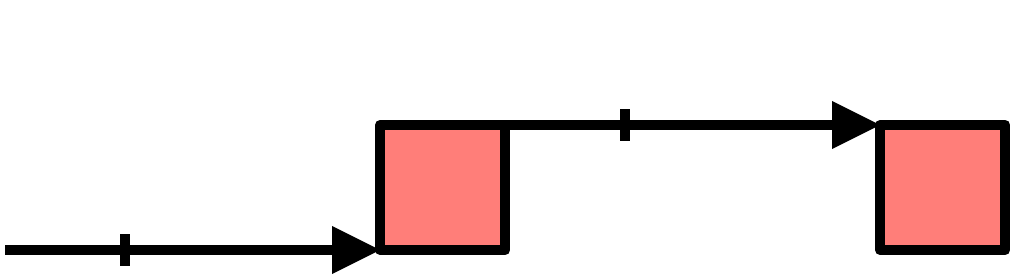}
    	\end{minipage}
& \begin{Large}$\frac{x^6y^2}{(1-x)^2}$\end{Large} \\ \hline
$M_3 = (rr^*)\square(r^*)v$
&
        \begin{minipage}{4cm}
        \centering
        \smallskip
	\includegraphics[height = 1cm]{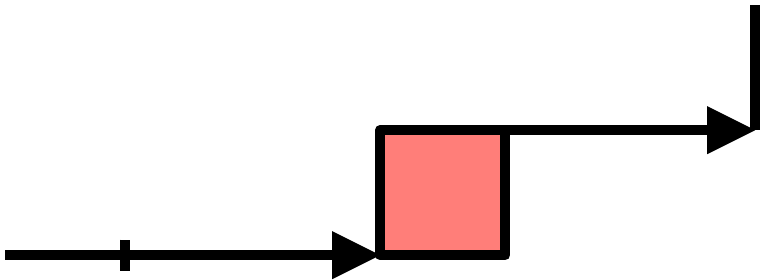}
    	\end{minipage}
& \begin{Large}$\frac{x^4y}{(1-x)^2}$\end{Large} \\ \hline
$M_4 =  (rr^*)\square(rr^*)v'$
&
        \begin{minipage}{4cm}
        \centering
        \smallskip
	\includegraphics[height = 1cm]{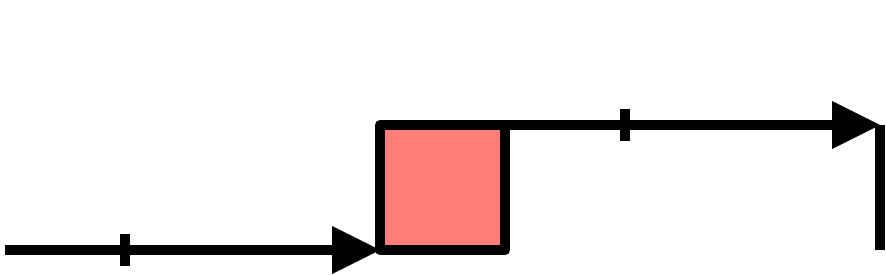}
    	\end{minipage}
& \begin{Large}$\frac{x^5y}{(1-x)^2}$\end{Large} \\ \hline
$M_5 = (rr^*)v(r^*)\square$
&
        \begin{minipage}{4cm}
        \centering
        \smallskip
	\includegraphics[height = 1cm]{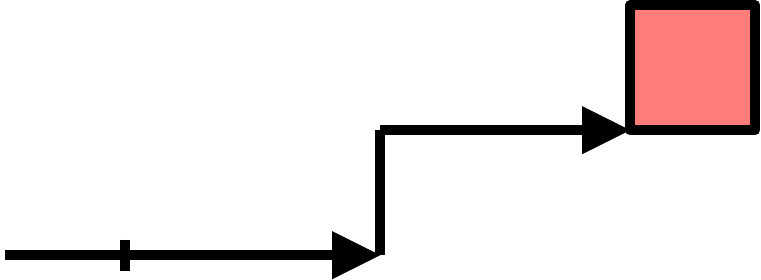}
    	\end{minipage}
& \begin{Large}$\frac{x^4y}{(1-x)^2}$\end{Large} \\ \hline
$M_6 = (rr^*)v(rr^*)\square'$
&
        \begin{minipage}{4cm}
        \centering
        \smallskip
	\includegraphics[height = 1cm]{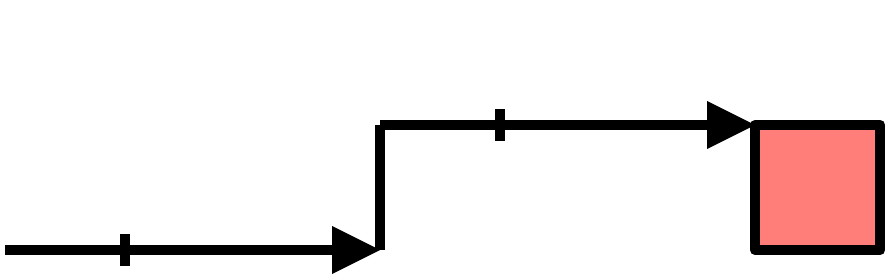}
    	\end{minipage}
& \begin{Large}$\frac{x^5y}{(1-x)^2}$\end{Large} \\ \hline
$M_7 = (rr^*)v(r^*)v$
&
        \begin{minipage}{4cm}
        \centering
        \smallskip
	\includegraphics[height = 1cm]{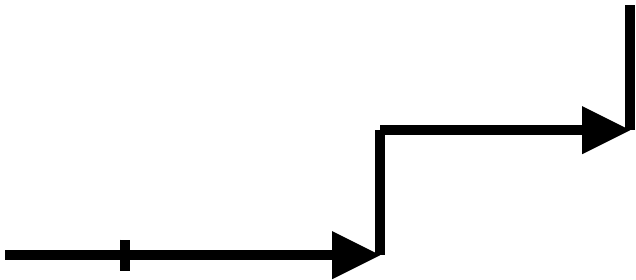}
    	\end{minipage}
& \begin{Large}$\frac{x^3}{(1-x)^2}$\end{Large}\\ \hline
$M_8 = (rr^*)v(rr^*)v'$
&
        \begin{minipage}{4cm}
        \centering
        \smallskip
	\includegraphics[height = 1cm]{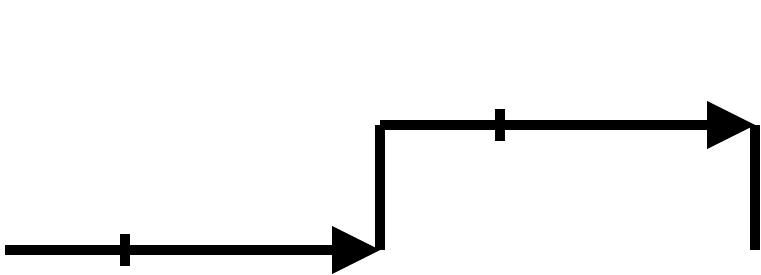}
    	\end{minipage}
& \begin{Large}$\frac{x^4}{(1-x)^2} $\end{Large} \\ \hline
  \end{tabular}
  \caption{
  The eight types of irreducible partial states and their generating functions.}
  \label{tab:movfuncgen}
\end{center}
\end{table}

\begin{proposition}\label{prop:M}
The generating function for the irreducible partial states $M$ is
\begin{eqnarray*}
  M(x,y) & = & \frac{x^3 + x^4 + 2x^4y + 2x^5y + x^5y^2 + x^6y^2}{(1-x)^2}.\\
\end{eqnarray*}
\end{proposition}

\begin{proof}
The word of an irreducible partial state has exactly two symbols that contribute a vertical move, which can be either a $v$ or a $\square$. Thus there are eight different families $M_1,\dots, M_8$, corresponding to the irreducible partial states of the following forms: 
\[
\begin{array}{cccc}
\dots \square \dots \square & \dots \square \dots \square' & \dots \square \dots v & \dots \square \dots v' \\
\dots v \dots \square & \dots v \dots \square' & \dots v \dots v & \dots v \dots v'  
\end{array}
\]
where $\square'$ and $v'$ represent a move whose vertical step is in the opposite direction to the previous vertical step. 
Table~\ref{tab:movfuncgen} illustrates these 8 families together with their corresponding generating functions.

Consider for example the family $M_2$ consisting of partial states of the form $\dots \square \dots \square' $. We must have at least one horizontal step before the first $\square$, and at least one horizontal step between the two $\square$s, to make sure they do not intersect. Therefore the partial states in $M_2$ are the words in the language $(rr^*)\square(rr^*)\square'$, whose generating function is
\[
\left(x \cdot \frac{1}{1-x}\right) x^2y \left(x \cdot \frac{1}{1-x}\right) x^2y = \frac{x^6y^2}{(1-x)^2}.
\]
The other formulas in Table~\ref{tab:movfuncgen} follow similarly. 
Finally,  
$M(x,y)$ is obtained by adding the eight generating functions in the table.
\end{proof}

\subsubsection{\textsf{Enumeration of irreducible final partial states}}
Now let us compute the generating function $F(x,y)$ for irreducible final 
partial states.

\begin{table}[h]
 \begin{center}
  \begin{tabular}{|l|c|c|} \hline
    Irreducible move & Illustration & Generating function \\ \hline
 $F_1 = r^*$		&
    \begin{minipage}{4cm}
    \centering
    \includegraphics[height = 1cm]{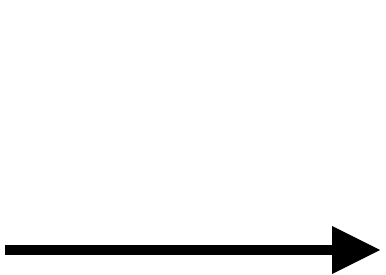}
    \end{minipage}
 &	\begin{Large}$\frac{1}{1-x}$\end{Large}\\ \hline

 $F_2= (rr^*)\llcorner$		&
    \begin{minipage}{4cm}
    \centering
    \includegraphics[height = 1cm]{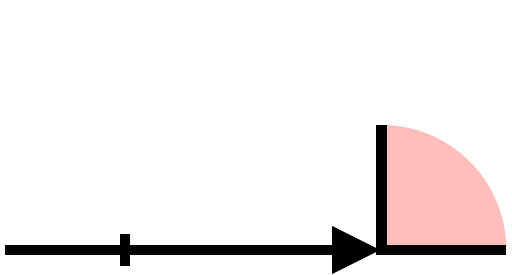}
    \end{minipage}
&	\begin{Large}$\frac{x^2y}{1-x}$\end{Large}\\ \hline

 $F_3=(rr^*)\square(r^*)$		&
    \begin{minipage}{4cm}
    \centering
    \includegraphics[height = 1cm]{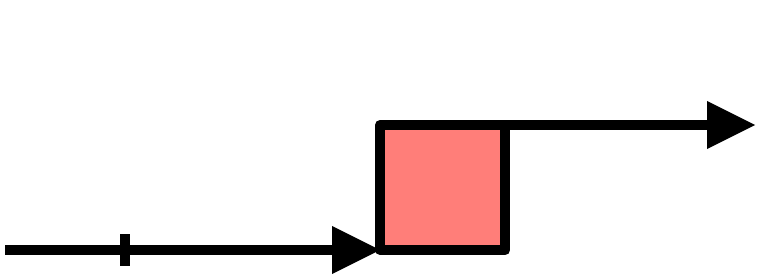}
    \end{minipage}
&	\begin{Large}$\frac{x^3y}{(1-x)^2}$\end{Large}\\ \hline
 
 $F_4= (rr^*)\square(r^*)\llcorner$		&
    \begin{minipage}{4cm}
    \centering
    \includegraphics[height = 1cm]{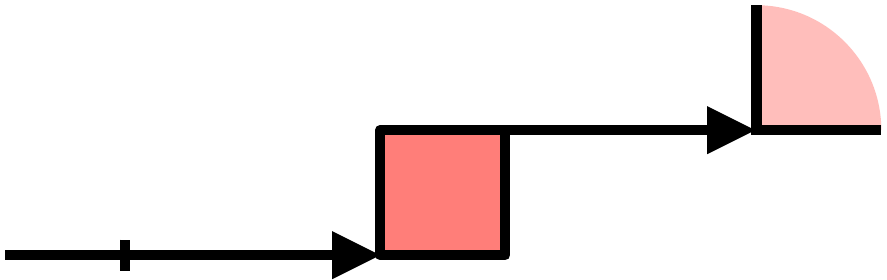}
    \end{minipage}
&	\begin{Large}$\frac{x^4y^2}{(1-x)^2}$ \end{Large}\\ \hline

 $F_5= (rr^*)\square(rr^*){\llcorner}'$		&
    \begin{minipage}{4cm}
    \centering
    \includegraphics[height = 1cm]{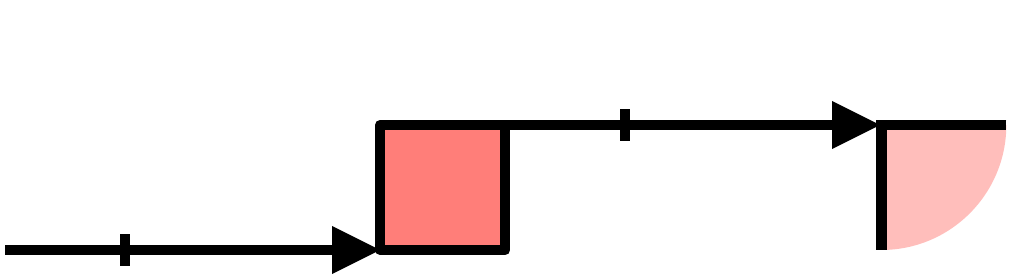}
    \end{minipage}
&	\begin{Large}$\frac{x^5y^2}{(1-x)^2}$\end{Large}\\ \hline

 $F_6= (rr^*)v(r^*)$		&
    \begin{minipage}{4cm}
    \centering
    \includegraphics[height = 1cm]{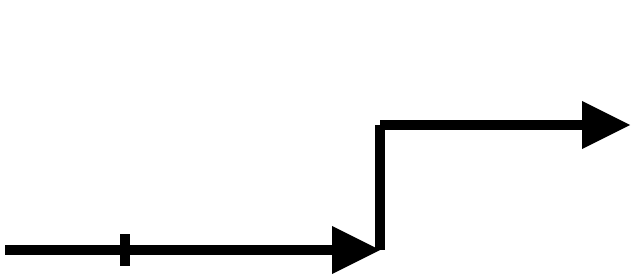}
    \end{minipage}
&	\begin{Large}$\frac{x^2}{(1-x)^2}$\end{Large}\\ \hline

 $F_7= (rr^*)v(r^*){\llcorner}$		&
    \begin{minipage}{4cm}
    \centering
    \includegraphics[height = 1cm]{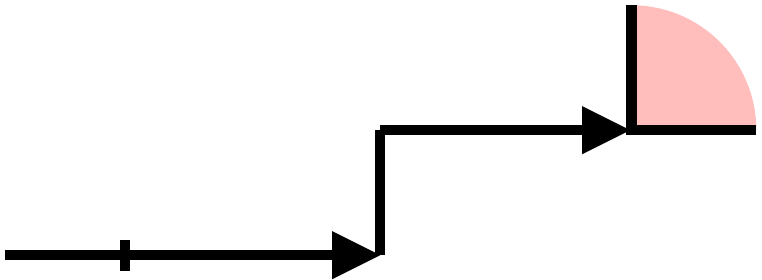}
    \end{minipage}
&	\begin{Large}$\frac{x^3y}{(1-x)^2}$\end{Large}\\ \hline

 $F_8= (rr^*)v(rr^*){\llcorner}'$		&
    \begin{minipage}{4cm}
    \centering
    \includegraphics[height = 1cm]{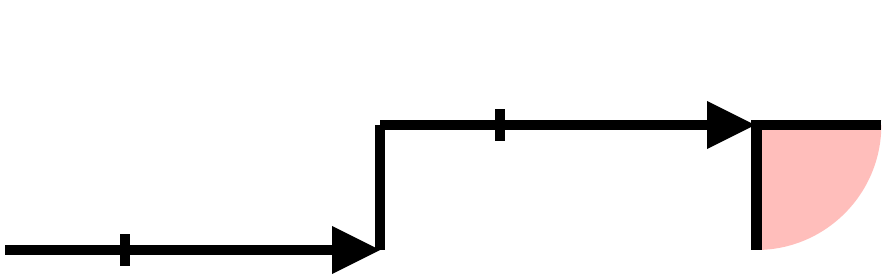}
    \end{minipage}
&	\begin{Large}$\frac{x^4y}{(1-x)^2}$\end{Large}\\ \hline

%
%
\end{tabular}
\caption{The eight types of irreducible final partial states and their generating functions.}\label{tab:final}
 \end{center}
\end{table}

\begin{proposition}\label{prop:F}
The generating function for the final irreducible partial states is
\begin{eqnarray*}
F(x,y) & = & \frac{1-x+x^2+x^2y+x^3y+x^4y+x^4y^2+x^5y^2}{(1-x)^2}.
\end{eqnarray*}
\end{proposition}

\begin{proof}
The word of each irreducible final partial state has at most one symbol among $\{v, \square\}$, and can possibly end with  $\llcorner$. Again, we let $\llcorner'$ represent a move whose vertical step is in the opposite direction to the previous vertical step. 
Thus there are eight different families $F_1,\dots, F_8$, corresponding to the irreducible partial states of the following forms: 
\[
\begin{array}{ccccc}
&\dots & & \dots \llcorner & \\
\dots v \dots && \dots v \dots \llcorner && \dots v \dots \llcorner'  \\
\dots \square \dots && \dots \square \dots \llcorner && \dots \square \dots \llcorner'  
\end{array}
\]
Table~\ref{tab:final} shows the eight different families of possibilities together with their corresponding generating functions.  
To obtain $F(x,y)$ we add their eight generating functions.
\end{proof}

\subsubsection{\textsf{The $f$-vector and Euler characteristic of the configuration space $\S_{2,n}$}}

\begin{reptheorem}{th:faceenum}
Let $\S_{2,n}$ be the configuration space for the robot of length $n$ moving in a rectangular tunnel of width $2$. If $c_{n,d}$ denotes the number of $d$-dimensional cubes in $\S_{2,n}$, then
\[
 C(x,y) = \sum_{n, d \geq 0} c_{n,d} \, x^ny^d  = 
\frac{1 + xy + x^2 + x^2y +  x^3 + 3x^3y +  x^3y^2 + 2x^4y + 2x^4y^2+x^5y^2}
 {1 - 2 x + x^2 - x^3 - x^4 - 2 x^4 y - 2 x^5 y - x^5 y^2 - x^6 y^2}.
%
%
\]
\end{reptheorem}

The above series starts: 
\begin{small}
\begin{eqnarray*}
 C(x,y) &=&   1 + x (y + 2) + x^2 (3y + 4) + x^3 (y^2 + 8y + 8)+ x^4 (4y^2 + 18y + 15) \\
 & &  + x^5 (11y^2 + 38y + 28) + x^6(y^3+30y^2+81y+53) + \ldots
\end{eqnarray*}
\end{small}
in agreement with Table \ref{tab:nestarm}.

\begin{proof}
This follows from Corollary \ref{cor:MF} and Propositions \ref{prop:M} and \ref{prop:F}.
\end{proof}

\begin{corollary}
The generating function counting the number $c_n$ of states of $\S_{2,n}$ is
\[
 \sum_{n \geq 0} c_n x^n =
 \frac{1 + x^2 + x^3}{1-2x+x^2-x^3-x^4}
  = 1+2x+4x^2+8x^3+15x^4 + 28x^5+53x^6+\cdots
\]
\end{corollary}

\begin{proof}
This is a straightforward consequence of Theorem~\ref{th:faceenum}, substituting $y=0$.
\end{proof}

\begin{reptheorem}{th:Euler}
The Euler characteristic of the configuration space $\S_{2,n}$ equals $1$.
\end{reptheorem}

\begin{proof}
The generating function for the Euler characteristic of $\S_{2,n}$ is 
\begin{eqnarray*}
\sum_{n \geq 0} \chi(\S_{2,n}) x^n &=&  \sum_{n \geq 0}\left(\sum_{d \geq 0} (-1)^d c_{n,d}\right) x^n = C(x,-1) \\
& = & \frac{1-x-x^3+x^5}{1-2x+x^2-x^3+x^4+x^5-x^6} \\
			& = & \frac{1}{1-x}  =  1+x+x^2+x^3+\ldots
\end{eqnarray*}
by Theorem \ref{th:faceenum}. All the coefficients of this series are equal to 1, as desired.
\end{proof}

\section{\textsf{The combinatorics of {CAT}(0) cube complexes}}
\label{sec:CAT(0)}

\subsection{\textsf{CAT(0) spaces}}

We now define CAT(0) spaces. Consider a metric space $X$ where every two points $x$ and $y$ can be joined by a path of length $d(x,y)$, known as a \emph{geodesic}.
Let $T$ be a triangle in $X$ whose sides are geodesics of lengths $a,b,c$, and let $T'$ be the triangle with the same lengths in the Euclidean plane. For any geodesic chord of length $d$ connecting two points on the boundary of $T$, there is a comparison chord between the corresponding two points on the boundary of $T'$, say of length $d'$. If $d \leq d'$ for any such chord in $T$, we say that $T$ is a \emph{thin triangle} in $X$. We say the metric space $X$ has \emph{non-positive global curvature} if every triangle in $X$ is thin.

\begin{figure}[h]
\centering
\includegraphics[scale=1]{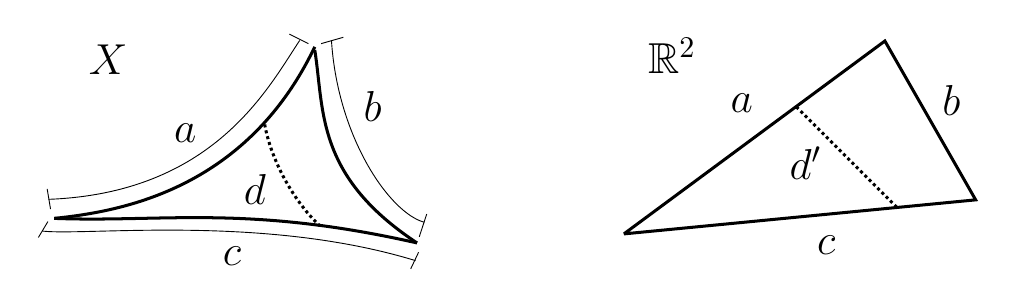}
\caption{A chord in a triangle in $X$, and the corresponding chord in the comparison triangle in $\mathbb{R}^2$. The triangle in $X$ is \emph{thin} if $d \leq d'$ for all such chords. \label{fig:thintriangle}}
\end{figure}

\begin{definition}\label{def:CAT(0)}
A metric space $X$ is said to be  $\mathrm{CAT}(0)$  if:

\noindent $\bullet$
 there is a unique geodesic (shortest) path between any two points in $X$, and

\noindent $\bullet$
 $X$ has non-positive global curvature. 

 \end{definition}

We are interested in proving that the cubical complex $\S_{m,n}$ is CAT(0), because of the following theorem:

\begin{theorem} \label{th:alg} \cite{AG, ABY, AOS, Re} 
If the configuration space of a robot is a $\mathrm{CAT}(0)$ cubical complex, there is an  algorithm to find the optimal way of moving the robot from one position to another.
\end{theorem}

%
%
%
%
%
%

\subsection{\textsf{CAT(0) cubical complexes}} 

It is not clear a priori how one might show that a cubical complex is CAT(0); Definition \ref{def:CAT(0)} certainly does not provide an efficient way of testing this property. 
Fortunately, \emph{for cubical complexes}, we know of two possible approaches.

\medskip

\noindent \textbf{\textsf{The topological approach.}}
The first approach uses Gromov's groundbreaking result that this subtle metric property has a topological--combinatorial characterization:

\begin{theorem}\label{th:Gromov}\cite{Gr}
A cubical complex is $\mathrm{CAT}(0)$ if and only if it is simply connected, and the link of every vertex is a flag simplicial complex.
\end{theorem}

Recall that a simplicial complex $\Delta$  is \emph{flag} if it has no empty simplices; more explicitly, if the $1$-skeleton of a simplex is in $\Delta$, then that simplex must be in $\Delta$. It is easy to see that, in the configuration space of a robot, every vertex has a flag link. Therefore, one approach to prove that these spaces are CAT(0) is to prove they are simply connected; see \cite{AG, GP} for examples of this approach.

\medskip

\noindent \textbf{\textsf{The combinatorial approach.}} We will use a purely combinatorial characterization of finite CAT(0) cube complexes ~\cite{AOS, Ro, Sa, W}; we use the formulation of Ardila, Owen, and Sullivant~\cite{AOS}. After drawing enough $\mathrm{CAT}(0)$ cube complexes, one might notice that they look a lot like distributive lattices. In trying to make this statement precise, one discovers a generalization of Birkhoff's Fundamental Theorem of Distributive Lattices; there are bijections: 
\begin{eqnarray*}
\textrm{distributive lattices} & \longleftrightarrow  & \textrm{posets}\\
\textrm{rooted CAT(0) cubical complexes} &  \longleftrightarrow & \textit{posets with inconsistent pairs} \textrm{ (PIPs)}
\end{eqnarray*}

Hence to prove that a cubical complex is a $\mathrm{CAT}(0)$ cubical complex, it suffices to choose a root for it, and identify the corresponding PIP. By the above correspondence, this PIP will serve as a ``certificate" that the complex is CAT(0). For a reasonably small complex, it is straighforward to identify the PIP, as described below. For an infinite family of configuration spaces like the one that interests us, one may do this for a few small examples and hope to identify a pattern that one can prove in general. Along the way, one gains a better understanding of the combinatorial structure of the complexes one is studying. 
This approach was first carried out in \cite{ABY} for two examples, and the goal of this paper is to carry it out for the robots $R_{m,n}$. This family is considerably more complicated than the ones in \cite{ABY}, in particular, because it appears to be the first known example where the skeleton of the complex is not a distributive lattice, and the PIP does indeed have inconsistent pairs.

\begin{figure}[h]
\begin{center}
\includegraphics[width=1.7in]{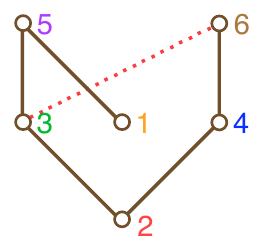} \qquad \qquad \qquad
\includegraphics[width=2.7in]{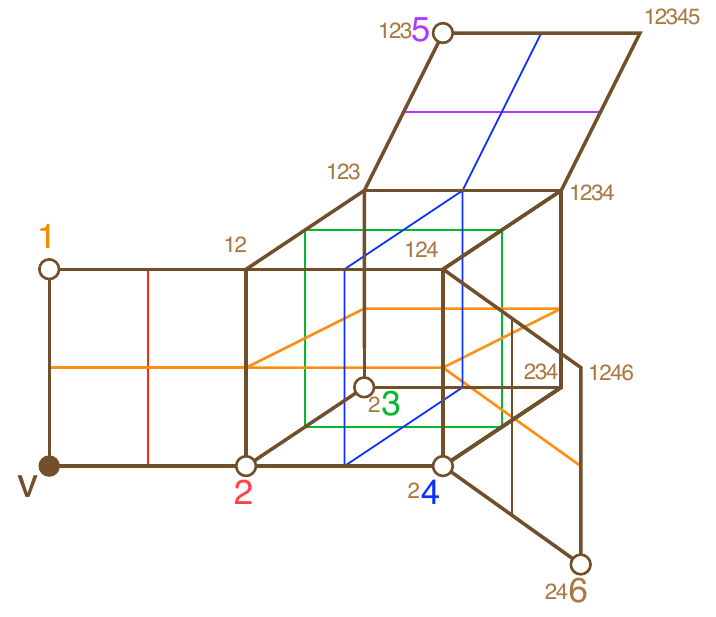}
\caption{A PIP and the corresponding rooted $\mathrm{CAT}(0)$ cubical complex.\label{fig:bijection} }
\end{center}
\end{figure}

\newpage

Let us describe this method in more detail.

\begin{definition} \cite{AOS, W}
A \emph{poset with inconsistent pairs (PIP)} is a poset $P$ together with a collection of \emph{inconsistent pairs}, which we denote $p \nleftrightarrow q$ (where $p \neq q$), such that

\begin{center}
if $p \nleftrightarrow q$ and $q < q'$ then $p \nleftrightarrow q'$.
\end{center}
\end{definition}

Note that PIPs are equivalent to \emph{event structures} \cite{W} in computer science, and are closely related to  \emph{pocsets} \cite{Ro, Sa} in geometric group theory.
The \emph{Hasse diagram} of a poset with inconsistent pairs (PIP) is obtained by drawing the poset, and connecting each $<$-minimal inconsistent pair with a dotted line. 
The left panel of Figure \ref{fig:bijection} shows an example.

\begin{theorem}\cite{AOS, Ro, Sa}\label{th:poset} 
There is a bijection $P \mapsto X(P)$ between posets with inconsistent pairs and rooted $\mathrm{CAT}(0)$ cube complexes.
\end{theorem}

It is useful to describe both directions of this bijection.

\bigskip

\noindent \textsf{Rooted CAT(0) cubical complex $\longmapsto$ PIP}: A CAT(0) cube complex $X$ has a system of \emph{hyperplanes} as described by Sageev \cite{Sa}. Each $d$-cube $C$ in $X$ has $d$ ``hyperplanes" of codimension 1, each one of which is a $(d-1)$-cube orthogonal to an edge direction of $C$ and splits it into two equal parts. When two cubes $C$ and $C'$ share an edge $e$, we identify the two hyperplanes in $C$ and in $C'$ that are orthogonal to $e$. The result of all these identifications is the set of hyperplanes of $X$. The right panel of Figure \ref{fig:bijection} shows a CAT(0) cube complex and its six hyperplanes.

Now we define the PIP $P$ associated to $X$. The elements of $P$ are the hyperplanes of $X$. For hyperplanes $H_1$ and $H_2$, we declare $H_1 < H_2$ if, starting at the root $v$ of $X$, one must cross hyperplane $H_1$ before one can cross hyperplane $H_2$. Finally, we declare $H_1 \nleftrightarrow H_2$ if, starting at the root $v$ of $X$, it is impossible to   cross both $H_1$ and $H_2$ without backtracking. The left panel of Figure \ref{fig:bijection} shows the PIP associated to the rooted complex of the right panel.

\bigskip

\noindent 
\textsf{PIP $\longmapsto$ rooted CAT(0) cubical complex}: Let $P$ be a PIP.
Recall that an \emph{order ideal} of $P$ is a subset $I$ such that if $x<y$ and $y \in I$ then $x \in I$. We say that $I$ is \emph{consistent} if it contains no inconsistent pair.

The vertices of $X(P)$ are identified with the consistent order ideals of $P$. 
There is a cube $C(I,M)$ for each pair $(I, M)$ of a consistent order ideal $I$ and a subset $M \subseteq I_{max}$, where $I_{max}$ is the set of maximal elements of $I$. This cube has dimension $|M|$, and its vertices are obtained by removing from $I$ the $2^{|M|}$ possible subsets of $M$. The cubes are naturally glued along their faces according to their labels. The root is the vertex corresponding to the empty order ideal.

This bijection is also illustrated in Figure \ref{fig:bijection}; the labels of the cubical complex on the right correspond to the consistent order ideals of the PIP on the left. 


\section{\textsf{The coral PIP, coral tableaux, and the proof of Theorem \ref{th:CAT(0)}}}
\label{sec:ourPIP}

We have now described all the preliminaries necessary to prove our main result that the configuration space $\S_{m,n}$ of the robotic arm in a tunnel is a $\mathrm{CAT}(0)$ cubical complex. We will achieve this by identifying the PIP corresponding to it under the bijection of Theorem~\ref{th:poset}. Interestingly, this PIP has some resemblance with Young's lattice of partitions. However, instead of partitions, its elements correspond to certain paths which we call \emph{coral snakes}.

\subsection{\textsf{The coral PIP}}

\begin{definition} A \emph{coral snake} $\lambda$ of height at most $m$ is a path of unit squares, colored alternatingly black and red (starting with black), inside the tunnel of width~$m$ such that:
\begin{enumerate}[(i)]
\item The snake $\lambda$ starts at the bottom left of the tunnel, and takes steps up, down, and right.
\item Suppose $\lambda$ turns from a vertical segment $V_1$ to a horizontal segment $H$ to a vertical segment $V_2$ at corners $C_1$ and $C_2$. Then $V_1$ and $V_2$ face the same direction if and only if $C_1$ and $C_2$ have the same color. (Note: We consider the first column of the snake a vertical segment going up, even if it consists of a single cell.)
\end{enumerate}
The \emph{length} $l(\lambda)$ is the number of unit squares of $\lambda$, the \emph{height} $h(\lambda)$ is the number of  rows it touches, and the \emph{width} $w(\lambda)$ is the number of columns it touches. We say that~$\mu$ contains~$\lambda$, in which case we write $\lambda \preceq \mu$, if $\lambda$ is an initial sub-snake of $\mu$ obtained by restricting to the first $k$ cells of $\mu$ for some $k$. We write $\lambda \prec \mu$ if $\lambda \preceq \mu$ and $\lambda \neq \mu$.
\end{definition}

Figure \ref{fig:snake} shows a coral snake $\lambda$ of length $l(\lambda) = 11$, height $h(\lambda) = 3$, and width $w(\lambda) = 7$. We encourage the reader to check condition $(ii)$. We often omit the colors of the coral snake when we draw them, since they are uniquely determined.

\begin{figure}[htbp]
	\centering
		\includegraphics[height = 2.2cm]{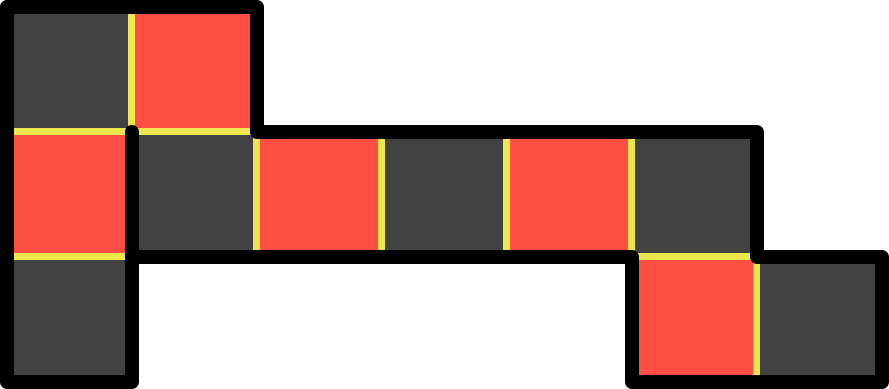} \qquad 
		\includegraphics[height = 2.5cm]{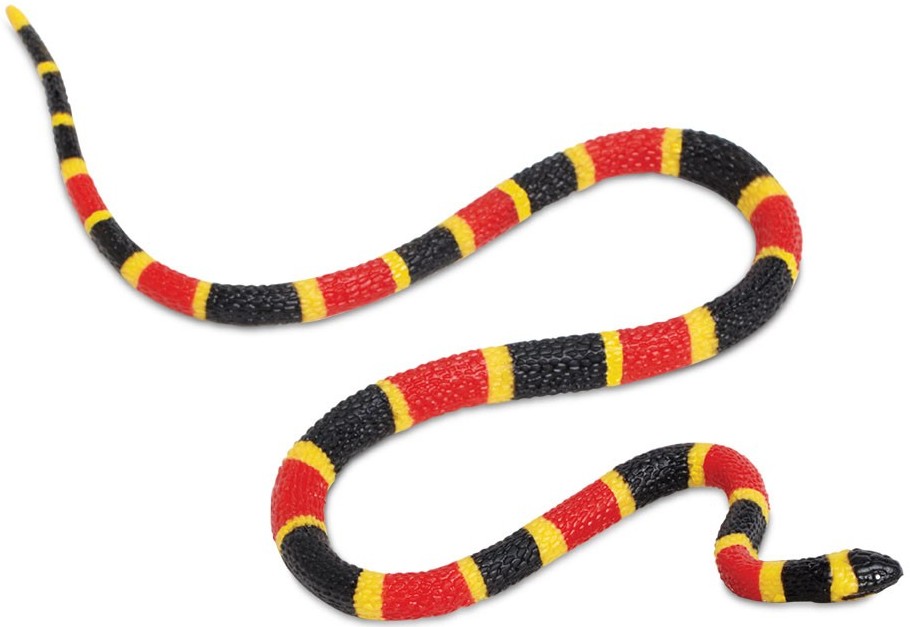}	\caption{A mathematical and (a photograph of) a  real-life coral snake.}
	\label{fig:snake}
\end{figure}

\begin{remark}
Our notion of containment of coral snakes differs from the notion of containment in the plane. For example, if $\lambda$ is the snake with two boxes corresponding to one step right, and $\mu$ is the snake with four boxes given by consecutive steps up-right-down, then $\lambda$ is contained in $\mu$ in the plane. However, $\lambda$ is not a sub-snake of $\mu$ and therefore $\lambda \npreceq \mu$. 
\end{remark}

\begin{definition}\label{def:snakePIP}
Define the \emph{coral PIP} $C_{m,n}$ as follows:

\noindent
$\bullet$ Elements: pairs $(\lambda,s)$ of a coral snake $\lambda$ with $h(\lambda) \leq m$ and a non-negative integer $s$ with $s \leq n-l(\lambda)-w(\lambda)+1$.

\noindent
$\bullet$ Order: 
$(\lambda,s) \leq (\mu,t)$ if $\lambda \preceq  \mu$ and $s \geq t$.

\noindent
$\bullet$ Inconsistency: $(\lambda,s)  \nleftrightarrow (\mu,t)$ if neither $\lambda$ nor $\mu$ contains the other.

For simplicity, we call the elements of the coral PIP \emph{numbered snakes}.
\end{definition}

\begin{remark}\label{rem:width2}
Figures \ref{fig:coralPIP} and \ref{fig:PIP9_withpaths} illustrate the coral PIPs $C_{m,2}$ for the tunnel of width 2. They have a nice self-similar structure\footnote{During a break, after weeks of unsuccessful attempts to describe these PIPs, the Pacific Ocean sent us a beautiful coral that looked just like them. The proverbial fractal structure of real-life corals helped us discover the self-similar nature of these PIPs; this led to their precise definition and inspired their name.}
in the following sense: they grow vertically supported on a main vertical ``spine", forming several sheets of roughly triangular shape. Along the way they grow other vertical spines, and every other spine supports a smaller coral PIP. The situation for higher $m$ is similar, though more complicated. 
\end{remark}

We will prove the following strengthening of Theorem \ref{th:CAT(0)}.

\begin{theorem}\label{thm:CAT}
The configuration space $\S_{m,n}$ of the robotic arm $R_{m,n}$ of length $n$ in a tunnel of width $m$ is a $\mathrm{CAT}(0)$ cubical complex. When it is rooted at the horizontal position of the arm, its corresponding PIP is the coral PIP $C_{m,n}$ of Definition~\ref{def:snakePIP}. 
\end{theorem}

\begin{figure}[tbh]
	\centering
		\includegraphics[height = 10cm]{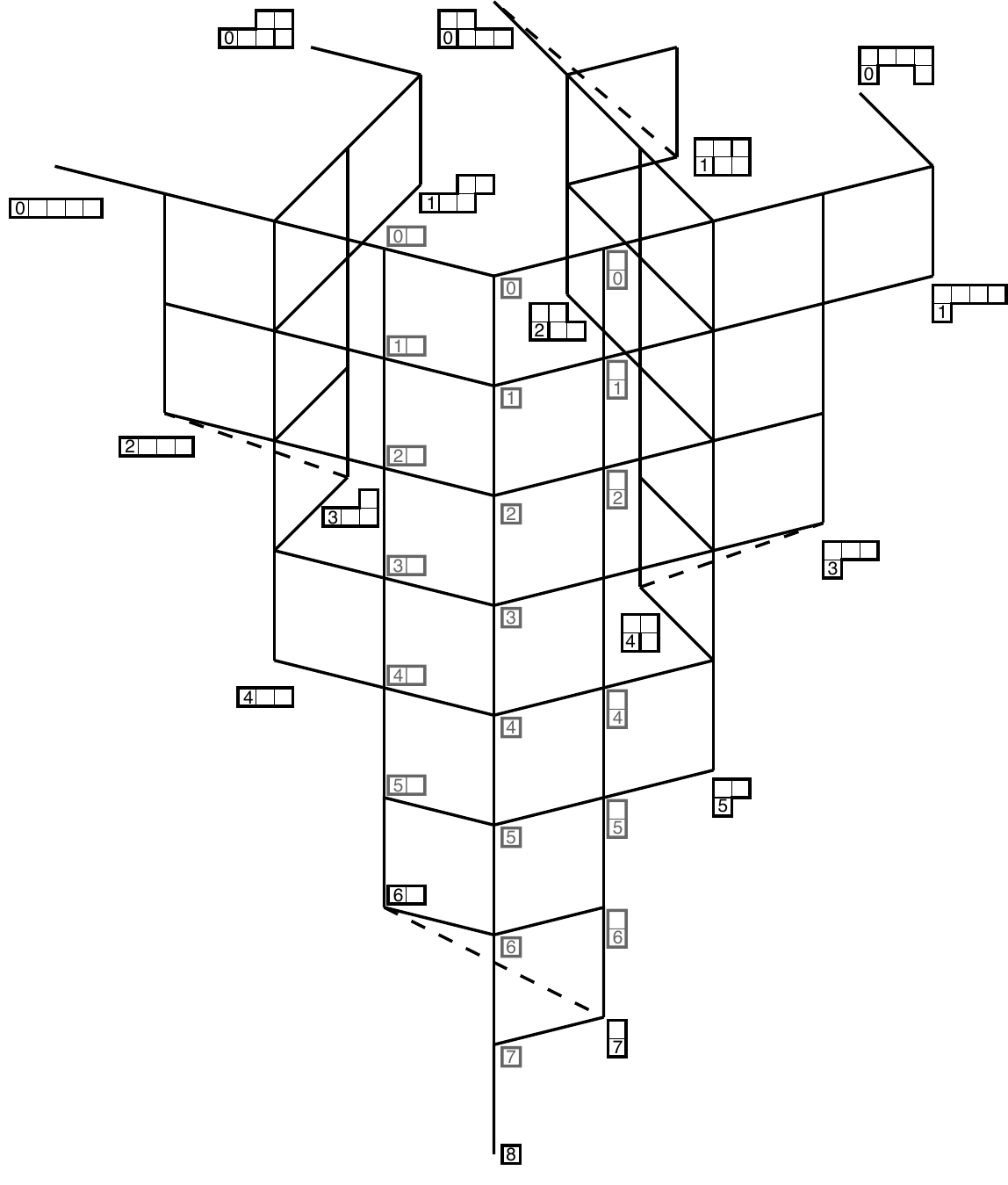}  
		\caption{The PIP $C_{m,n}$ for $m=2$ and $n=9$. Some of the labels are omitted; to obtain them, note that each vertical column consists of the elements $(\lambda, s)$ for a fixed shape $\lambda$ and $0 \leq s \leq n-l(\lambda)-w(\lambda)+1$, listed in decreasing order with respect to $s$.}
	\label{fig:PIP9_withpaths}
\end{figure}

\subsection{\textsf{Coral tableaux}}

Although we did not need to mention them explicitly in the statement of Theorem \ref{thm:CAT}, certain kinds of tableaux played a crucial role in its discovery, and are indispensible in its proof.

\begin{definition}\label{def:coralsnaketableau}
A \emph{coral snake tableau} (or simply \emph{coral tableau}) $T$  on a coral snake $\lambda$ is a filling of the squares of $\lambda$ with non-negative integers which are strictly increasing horizontally and weakly increasing vertically, following the direction of the snake. 
We call $\lambda=\sh(T)$ the \emph{shape} of $T$, and we say that $T$ is of \emph{type} $(m,n)$ if $h(\lambda) \leq m$ and $\max(T) + l(\lambda) \leq n$.
\end{definition}

Note that if $T$ is of type $(m,n)$, it is also of type $(m',n')$ for any $m' \geq m$ and $n' \geq n$.

\begin{definition}
We call a coral tableau \emph{tight} if its entries are constant along columns and increase by one along rows.   
\end{definition}

\begin{figure}[htbp]
	\centering
		\includegraphics[height = 2.2cm]{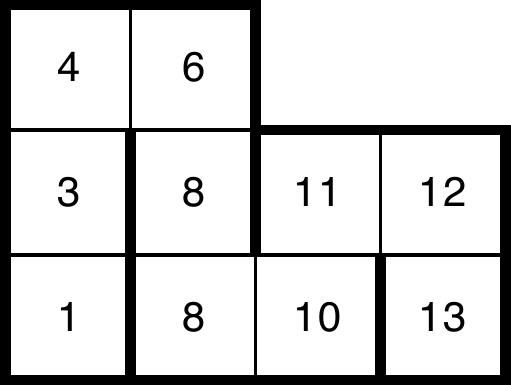} \qquad 
				\includegraphics[height = 2.2cm]{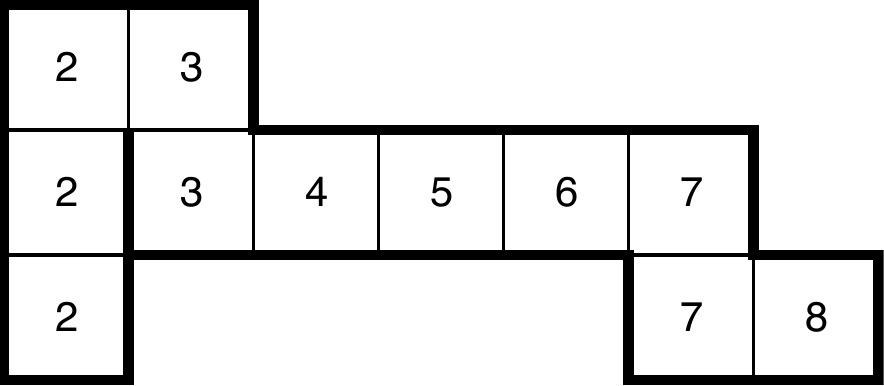}
	\caption{Two coral tableaux; the one on the right is tight.}
	\label{fig:snaketableaux}
\end{figure}

\begin{lemma}\label{lemma:tightstatestotableau}
There is a bijection between tight coral tableaux of type $(m,n)$ and the numbered coral snakes of the PIP $C_{m,n}$.
\end{lemma}

\begin{proof}
A tight coral tableau $T$ is uniquely determined by its shape and first entry, so the bijection is given by sending $T$ to $(\sh(T), \min(T))$; it suffices to observe that $\max(T) = \min(T) + w(\lambda)-1$ when $T$ is tight, so the inequality $\min(T) \leq n- l(\lambda)-w(\lambda)+1$ holds if and only if $\max(T) + l(\lambda) \leq n$.
\end{proof}

For example, the tight tableau on the right of Figure \ref{fig:snaketableaux} corresponds to $(\lambda, 2)$ where $\lambda$ is the snake of Figure \ref{fig:snake}.

\begin{lemma}\label{lemma:statestotableau}
The possible states of the robotic arm $R_{m,n}$ are in  bijection with the coral tableaux of type $(m,n)$. 
\end{lemma}

\begin{proof}
We encode a position $P$ of the robotic arm as a coral tableau $T$, building it up from left to right as follows. Every time that the robot takes a vertical step in row $i$, we add a new square to row $i$ of the tableau, and fill it with the number of the column the step was taken in. An example is shown in Figure~\ref{fig:snake_from_state}.   

It is clear that the entries of $T$ increase weakly in the vertical direction, and increase strictly in the horizontal direction. It remains to check that the snake $\sh(T)$ is indeed a coral snake. Suppose $\sh(T)$ goes from a vertical $V_1$ to a horizontal $H$ to a vertical $V_2$, turning at corners $C_1$ and $C_2$ of row $i$. Assume for definiteness that $V_1$ points up and $C_1$ is black. Then the black and red entries on $H$ represent up and down steps that the robot takes on row $i$; so the direction of $V_2$ is determined by the color of $C_2$ as desired. \end{proof}

\begin{figure}[htbp]
	\centering
		\includegraphics[height = 2.1cm]{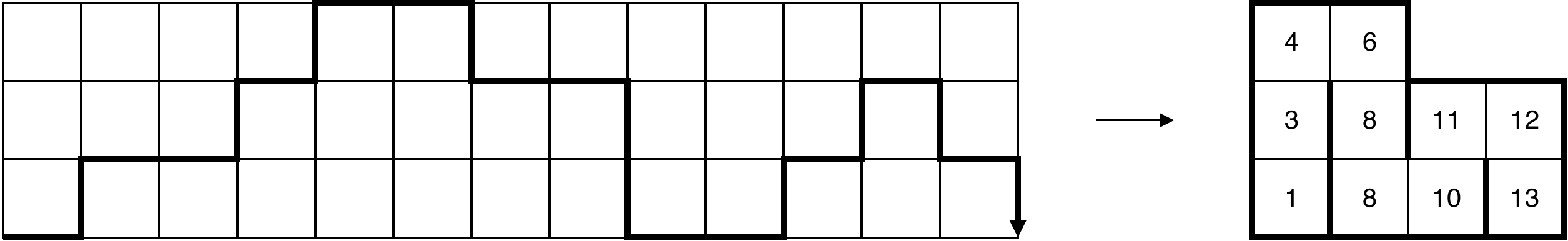}
	\caption{From a state $R$ to a coral tableau $T$.}
	\label{fig:snake_from_state}
\end{figure}

\begin{figure}[htbp]
	\centering
		\includegraphics[height = 2.1cm]{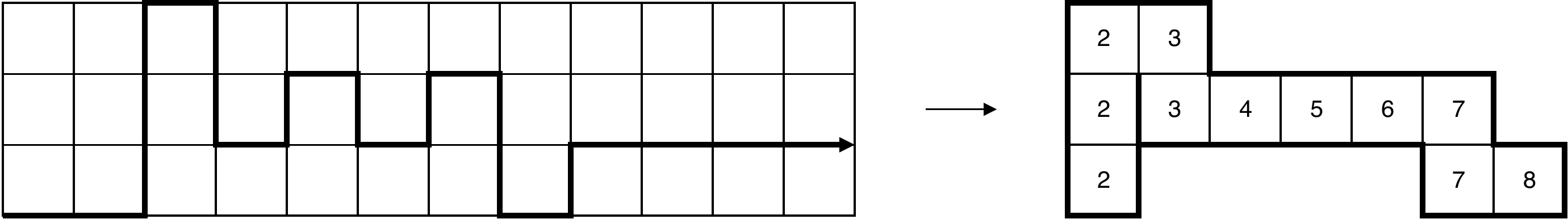}
	\caption{A tight state and its corresponding tight snake.}
	\label{fig:snake_tight}
\end{figure}

We call a state of the robot \emph{tight} if its corresponding snake tableau is tight, see Figure~\ref{fig:snake_tight}.

\subsection{\textsf{Proof of our main structural theorem}}

We are now ready to prove our strengthening of the main result of this paper, Theorem \ref{th:CAT(0)}. 

\begin{reptheorem}{thm:CAT}
The configuration space $\S_{m,n}$ of the robotic arm $R_{m,n}$ of length $n$ in a tunnel of width $m$ is a $\mathrm{CAT}(0)$ cubical complex. When it is rooted at the horizontal position of the arm, its corresponding PIP is the coral PIP $C_{m,n}$ of Definition~\ref{def:snakePIP}. 
\end{reptheorem}

\begin{proof}
We need to show that, when rooted at the horizontal position, $\S_{m,n}$ is the cubical complex $X(C_{m,n})$ associated to the PIP $C_{m,n}$ under the bijection of Theorem \ref{th:poset}. We proceed in three steps. 

\smallskip
\noindent \textsf{Step 1}. Decomposing $\S_{m,n}$ and $C_{m,n}$ by shape. 
\smallskip

For a coral snake $\lambda$ of type $(m,n)$ let $\S_{m,n}^\lambda$ be the induced subcomplex of $\S_{m,n}$ whose vertices are the coral tableaux $T$ with $\sh(T) \preceq \lambda$.
Similarly, let $C^\lambda_{m,n}$ be the subPIP of $C_{m,n}$ consisting of the pairs $(\mu, s)$ such that $\mu \preceq \lambda$. Notice that $C^\lambda_{m,n}$ is a poset which has no inconsistent pairs.
We have 
\begin{equation} \label{decompintolambdas}
\S_{m,n} = \bigcup_{\lambda \textrm{ of type } (m,n)} \S_{m,n}^\lambda
\qquad \qquad 
C_{m,n} = \bigcup_{\lambda \textrm{ of type } (m,n)} C_{m,n}^\lambda
\end{equation}
where each (non-disjoint) union is over the coral tableaux $\lambda$ of type $(m,n)$. (Of course, since $\lambda \preceq \lambda'$ implies $\S_{m,n}^\lambda \subseteq \S_{m,n}^{\lambda'}$ and $C_{m,n}^\lambda \subseteq C_{m,n}^{\lambda'}$, it is sufficient to take the union over those $\lambda$ which are maximal under inclusion.)

\smallskip
\noindent \textsf{Step 2}. Showing $\S_{m,n}^\lambda = X(C_{m,n}^\lambda)$ for each shape $\lambda$. 
\smallskip

We begin by establishing a bijection between the sets of vertices of both complexes.
The vertices of $X(C_{m,n}^\lambda)$ correspond to the consistent order ideals of the coral PIP $C_{m,n}^\lambda$. Since $C_{m,n}^\lambda$ has no inconsistent pairs, these are \emph{all} the order ideals of $C_{m,n}^\lambda$, which are the elements of the distributive lattice $J(C_{m,n}^\lambda)$. 

The vertices of $\S_{m,n}^\lambda$ correspond to the coral tableaux $T$ of type $(m,n)$ with shape $\sh(T) \preceq \lambda$ by Lemma \ref{lemma:statestotableau}. Let us extend each such tableau $T$ to a tableau of shape $\lambda$ by adding entries equal to $\infty$ on each cell of $\lambda - \sh(T)$. We may then identify the set of vertices $V(\S_{m,n}^\lambda)$ with the resulting set of \emph{extended coral tableaux} of shape $\lambda$ whose finite entries $x$ satisfy $x + l(\lambda) \leq n$.

This set $V(\S_{m,n}^\lambda)$ of extended $\lambda$-tableaux forms a poset under reverse componentwise order. Note that if $T_1$ and $T_2$ are in $V(\S_{m,n}^\lambda)$ then the componentwise maximum $T_1 \wedge T_2$ and the componentwise minimum $T_1 \vee T_2$ are also in $V(\S_{m,n}^\lambda)$. Then the meet $\wedge$ and join $\vee$ make $V(\S_{m,n}^\lambda)$ into a lattice. In fact, the definitions of $\wedge$ and $\vee$ imply that this lattice is distributive. Birkhoff's Fundamental Theorem of Distributive Lattices then implies that  $V(\S_{m,n}^\lambda) \cong J(P)$ where $P$ is the subposet of join-irreducible elements of $V(\S_{m,n}^\lambda)$. One easily verifies that $P$ consists precisely of the tight tableaux of type $(m,n)$ and shape $\lambda$, so $P \cong C_{m,n}^\lambda$. It follows that the coral tableaux $T$ of type $(m,n)$ with shape $\sh(T) \preceq \lambda$ are indeed in bijection with the consistent order ideals of the coral PIP $C_{m,n}^\lambda$.

We can make the bijection more explicit. Given a coral tableau $T$ of shape $\mu \preceq \lambda$, we can write $T=T_1 \vee T_2 \vee \cdots \vee T_{l(\mu)}$ as follows. For  each $i$ let $\mu_i$ be the subsnake consisting of the first $i$ boxes of $\mu$, and let $T_i$ be the unique tight tableau of shape $\lambda_i$ whose $i$th entry is equal to the $i$th entry of $T$. In fact we can reduce this to a minimal equality
\[
T = \bigvee_{i \textrm{ jump in } T} T_i
\]
where we say that $T$ \emph{jumps} at cell $i$ if $T$ remains a coral tableau of type $(m,n)$ when we increase its $i$th entry by $1$. The set $A(T) = \{T_i \, : \, i \textrm{ jump in } T\}$ is an antichain in $C_{m,n}^\lambda$, and the bijection above maps the coral tableau $T$ to the order ideal $I(T) \subseteq C_{m,n}^\lambda$ whose set of maximal elements is $I(T)_{max} = A(T)$. 

\begin{figure}[htbp]
	\centering
		\includegraphics[height = 1.3cm]{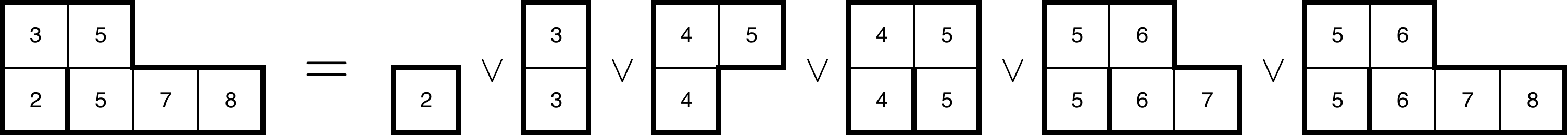}
	\caption{A coral tableau $T = T_1 \vee \cdots \vee T_6 = T_1 \vee T_2 \vee T_4 \vee T_6$ as a join of irreducibles.}
	\label{fig:tableau_as_a_join}
\end{figure}

Having established the bijection between the vertices of $\S_{m,n}^\lambda$ and $ X(C_{m,n}^\lambda)$, we now prove the isomorphism of these cubical complexes. Each cube $C(I,M)$ of $X(C_{m,n}^\lambda)$ is given by an order ideal $I \subseteq C_{m,n}^\lambda$ and a subset $M \subseteq I_{max}$. Let $T$ be the coral tableau corresponding to $I$, and $P$ the corresponding position of the robotic arm. Then $I_{max}$ corresponds to the set of ``descending" moves that can be performed at position $P$ to bring it closer to the horizontal position; namely, those of the form $\ulcorner \mapsto  \lrcorner$, $\llcorner \mapsto  \urcorner$, or flipping the end from a vertical position to a horizontal one facing right.

%
Any subset $M$ of those moves can be performed simultaneously, so this subset corresponds to a cube in $\S_{m,n}^\lambda$. One may check that every cube of $\S_{m,n}^\lambda$ arises in this way from a cube in $X(C_{m,n}^\lambda)$, and that the cubical complex structure is the same for both complexes.

\begin{figure}[htbp]
	\centering
		\includegraphics[height = 1.3cm]{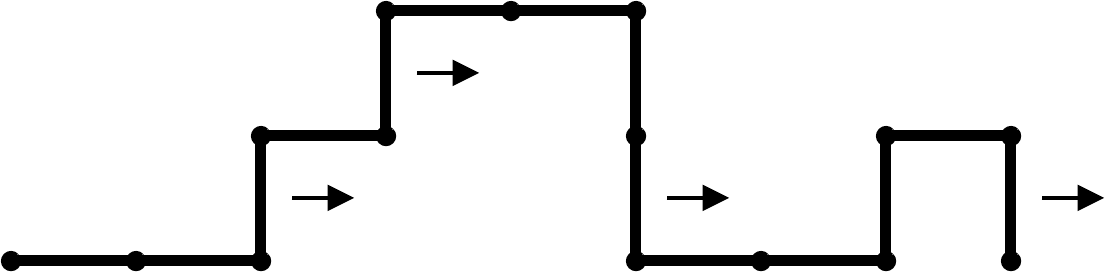}
	\caption{The cube corresponding to the moves above.}
	\label{fig:cube_as_moves}
\end{figure}

\smallskip
\noindent \textsf{Step 3}. Showing $\S_{m,n} = X(C_{m,n})$. 
\smallskip

Recall that, as $\lambda$ ranges over the shapes of type $(m,n)$, the subcomplexes $\S_{m,n}^\lambda$ cover the configuration space $\S_{m,n}$, and the subPIPs $C_{m,n}^\lambda$ (which happen to have no inconsistent pairs) cover the PIP $C_{m,n}$ by (\ref{decompintolambdas}). We now claim that the analogous statement  holds for the $\mathrm{CAT}(0)$ cube complex $X(C_{m,n})$ as well:
\begin{equation}\label{decompintolambdas2}
X(C_{m,n}) = \bigcup_{\lambda \textrm{ maxl. of type } (m,n)} X(C_{m,n}^\lambda).
\end{equation}
To see this, recall that each vertex $v$ of $X(C_{m,n})$ corresponds to a consistent order ideal $I=\{(\lambda_1, s_1), \ldots, (\lambda_k, s_k)\} \subseteq C_{m,n}$. Since $I$ is consistent, one of $\lambda_i$ and $\lambda_j$ contains the other one for all $i$ and $j$, and therefore the maximum shape $\lambda$ among them contains them all, and is of type $(m,n)$. It follows that $v$ is a vertex of $X(C_{m,n}^\lambda)$. Similarly, for any cube $C(I,M)$ of $X(C_{m,n})$, the consistent order ideal $I$ corresponds to a vertex $v$ in some $X(C_{m,n}^\lambda)$, and this means that $C(I,M)$ is in $X(C_{m,n}^\lambda)$ as well. 

Since $\S_{m,n}^\lambda = X(C_{m,n}^\lambda)$ for all $\lambda$, the last necessary step is to check that the decomposition $\S_{m,n} = \bigcup_\lambda \S_{m,n}^\lambda$ is compatible with the decomposition $X(C_{m,n}) = \bigcup_\lambda X(C_{m,n}^\lambda)$. 
This follows from the fact that for any $\lambda$ and $\mu$ we have
\[
\S_{m,n}^\lambda \cap \S_{m,n}^\mu = \S_{m,n}^\nu, \qquad 
X(C_{m,n}^\lambda) \cap X(C_{m,n}^\mu) = X(C_{m,n}^\nu);
\]
where $\nu = \lambda \wedge \mu$ is the largest coral snake which is less than both $\lambda$ and $\mu$.
\end{proof}

\begin{remark}
One might hope that Theorem~\ref{thm:CAT} could be generalized for a robotic arm moving in a $d$-dimensional tunnel $[0,n_1]\times \dots \times [0,n_d]$. However, the resulting cubical complexes are not $\mathrm{CAT}(0)$ in general. Even in the simplest 3-dimensional case $[0,n]\times [0,1]\times [0,1]$ the result does not generalize. Figure~\ref{fig:robots_3D} illustrates two examples of the cubical complexes for $n=1$ and $n=2$. The case $n=1$ consists of three vertices (states) and three edges (allowable moves) forming a triangle (without the interior face), while the case $n=2$ consists of seven vertices and eight edges forming an hexagon and a triangle glued together along an edge. In both cases, these cubical complexes are not CAT(0). For instance, they have non-contractible loops and there is not a unique geodesic connecting two opposite vertices of the hexagon.    

\begin{figure}[htbp]
	\centering
		\includegraphics[width = \textwidth]{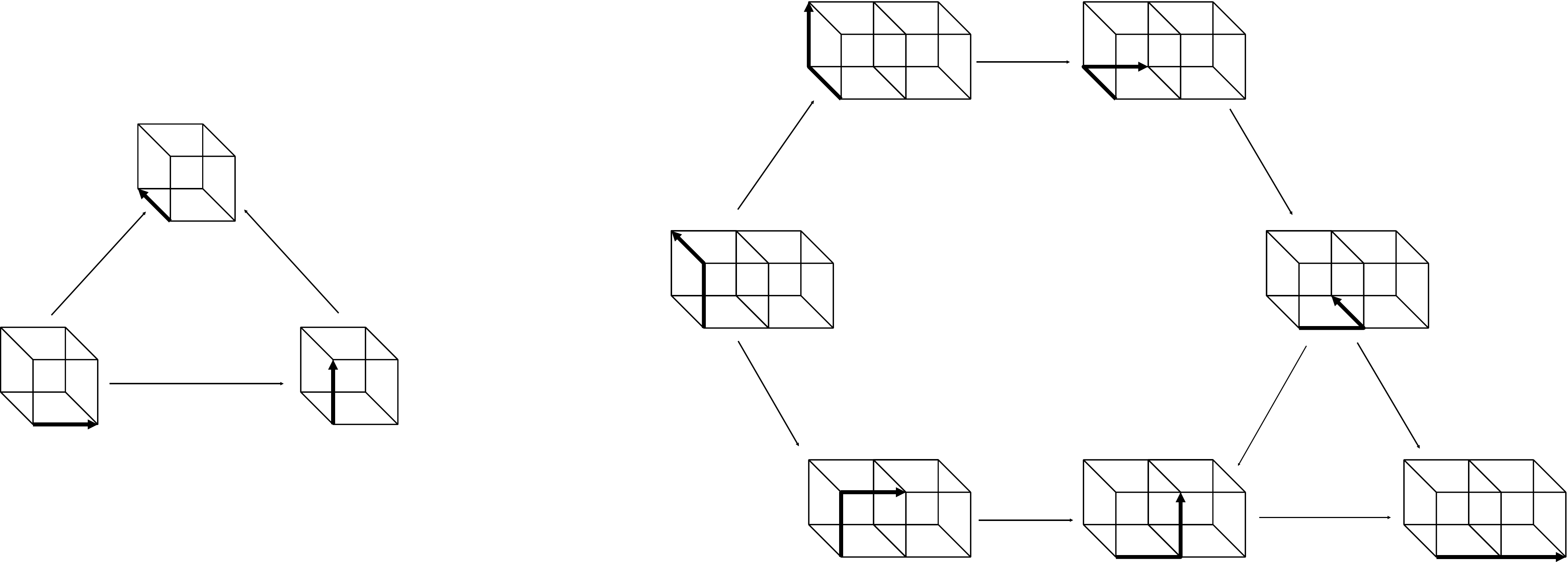}
	\caption{The cubical complex of a robotic arm in a 3-dimensional tunnel $[0,n]\times [0,1]\times [0,1]$ is not CAT(0) for $n=1$ or $n=2$.}
	\label{fig:robots_3D}
\end{figure}
  
\end{remark}

\section{\textsf{Shortest path and the diameter of the transition graph}}\label{sec:diameter}
The coral tableau representation turns out to be very useful for finding the distance between any two possible states of the robot in the transition graph $G(R_{m,n})$, as well as for finding its diameter. Before carrying this program out in detail, let us  briefly describe the intuition behind it.

By Theorem \ref{thm:CAT}, a position of the robot corresponds to a consistent order ideal in the coral PIP $C_{m,n}$  of Figures \ref{fig:coralPIP} and \ref{fig:PIP9_withpaths}. As we mentioned in Remark \ref{rem:width2} and these figures illustrate, these PIPs are obtained by gluing several \emph{sheets} of roughly triangular shape along some vertical spines; 
in the proof of Theorem \ref{thm:CAT}, 
these sheets are the subposets $C_{m,n}^\lambda$ for the maximal $\lambda$. Every such sheet grows out of the main spine, possibly branching along the way.
A careful look at the inconsistent pairs shows that any consistent order ideal in this PIP must be contained within a single sheet. 

Now suppose we wish to find the shortest path between two states $P$ and $P'$ of the robot. This is equivalent to finding the shortest way of transforming an order ideal $I$ in a sheet $T$ of the PIP into an order ideal $I'$ in another  sheet $T'$. To do this, one must first shrink the ideals $I$ and $I'$ into ideals $J \subseteq I $ and $J' \subseteq I'$ which lie in the intersection $T \cap T'$ of the sheets, and then find the best way of transforming $J$ into $J'$ inside $T \cap T'$. The following definitions make this more precise.

\newpage

\begin{definition}
Let $P$ and $P'$ be two positions of the robotic arm $R_{m,n}$, and $\lambda$ and $\lambda'$ be the shapes of their corresponding coral tableaux. Let $\lambda \wedge \lambda'$ be the largest coral snake contained in both $\lambda$ and $\lambda'$. 
We label the links of $P$ in decreasing order from $n$ to $1$ along the shape of the robot. The \emph{vertical labelling} of $P$ is the vector that reads only the labels of the vertical links of $P$ in the order they appear. This vector can be decomposed into two parts $v\circ w$, where $v$ consists of the labels  of the vertical links in $\lambda \wedge \lambda'$ and $w$ consists of the labels  of the vertical links in the complement $\lambda \backslash(\lambda \wedge \lambda')$.
 We call $v\circ w$ the \emph{$(P,P')$-decomposition of $P$}. Similarly, we also have a $(P',P)$-decomposition $v'\circ w'$ of the vertical labelling of $P'$.


\medskip

\begin{figure}[htbp]
	\centering
	\begin{overpic}[width=0.95\textwidth]{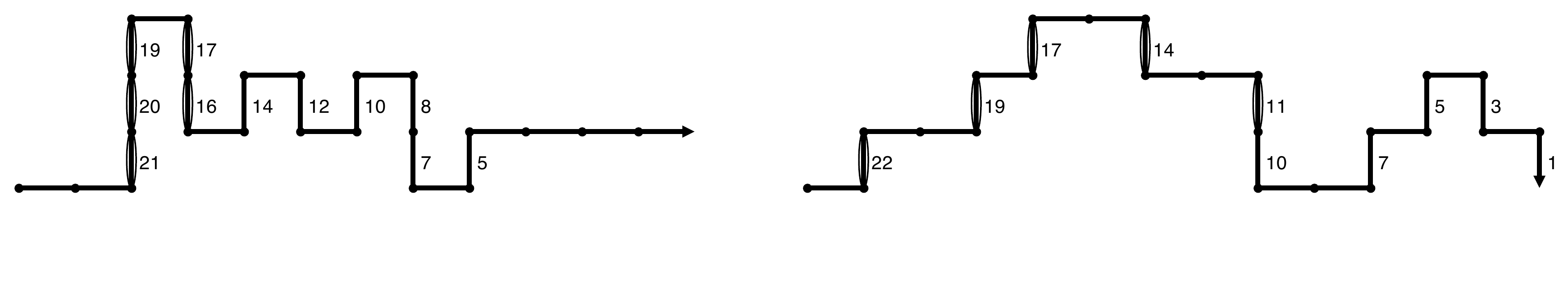}
		\put(22,19){\footnotesize $P$}
		\put(11,2.8){\footnotesize $v=(21,20,19,17,16)$}
		\put(11,0){\footnotesize $w=(14,12,10,8,7,5)$}
		\put(75,19){\footnotesize $P'$}
		\put(65,2.8){\footnotesize $v'=(22,19,17,14,11)$}
		\put(65,0){\footnotesize $w'=(10,7,5,3,1)$}
	\end{overpic}
	\caption{The vertical labellings of two positions of length 23
	and their decompositions.}
	\label{fig:robots_vertical_decomposition}
\end{figure}

\end{definition}

Figure~\ref{fig:robots_vertical_decomposition} illustrates this definition for the positions $P$ and $P'$ of the robot of length 23 in width 3 shown in Figures \ref{fig:snake_tight} and \ref{fig:snake_from_state}, respectively; the snakes $\lambda, \lambda',$ and $\lambda \wedge \lambda'$ are the shapes of the coral tableaux $T,T'$, and $T_1$ is Figure~\ref{fig:snake_decomposition} respectively. The smallest number of moves to get from position $P$ to position $P'$ is $d(P,P')=94$, as predicted by the following theorem.

\begin{proposition}\label{prop:distance}
Let $P$ and $P'$ be two positions of the robotic arm $R_{m,n}$. The distance between $P$ and $P'$ in the transition graph $G(R_{m,n})$ is equal to 
\begin{equation}\label{eq:distance}
d(P,P') = |w| + |v-v'| + |w'|,
\end{equation}
where $v\circ w$ is the $(P,P')$-decomposition of $P$,
$v'\circ w'$ is the $(P',P)$-decomposition of $P'$, and $|\cdot|$ denotes the $l_1$-norm, that is, $|(a_1, \ldots, a_k)| = |a_1| + \cdots + |a_k|$.
\end{proposition}

We will prove this proposition in Section~\ref{sec:distance_proof}. We will also use it to find an explicit formula for the diameter of the transition graph $G(R_{m,n})$.

If the width of the tunnel is~$m=1$ one can easily find the two positions of the robot that are at maximum distance from each other. We use the letters $u,r$ and $d$ for links pointing up, right and down respectively. 

\begin{lemma}
For~$m=1$, the maximum distance between two positions of the robot in the transition graph $G(R_{1,n})$ is attained by the pair $(urdrurdr\dots, rr\dots r)$ of \emph{left justified} and \emph{fully horizontal} robotic arms. The diameter of $G(R_{1,n})$ is
\[
\diam G(R_{1,n})=
\left\{ \begin{array}{rcl}
         \frac{n(n+2)}{4} & \mbox{for} & n\text{ even} \\ 
	\frac{(n+1)^2}{4}  & \mbox{for} & n\text{ odd}.
                \end{array}\right.
\]
\end{lemma}

For larger $m$ the situation is rather different.
Define the \emph{left justified robot} $\Left$, the \emph{shifted left justified robot} $\Leftt$, 
 and the \emph{fully horizontal robot} $\Hor$ of length~$n$ as
\[
\Left=u^mrd^mru^mr\dots, \quad \Leftt=(urdr)(u^mrd^mru^mr\dots), \quad \Hor=rrr\dots,
\]
respectively. The first two of these are illustrated in Figure~\ref{fig:robot_left}. 

%

\begin{figure}[htbp]
	\centering
	\includegraphics[height=2.2cm]{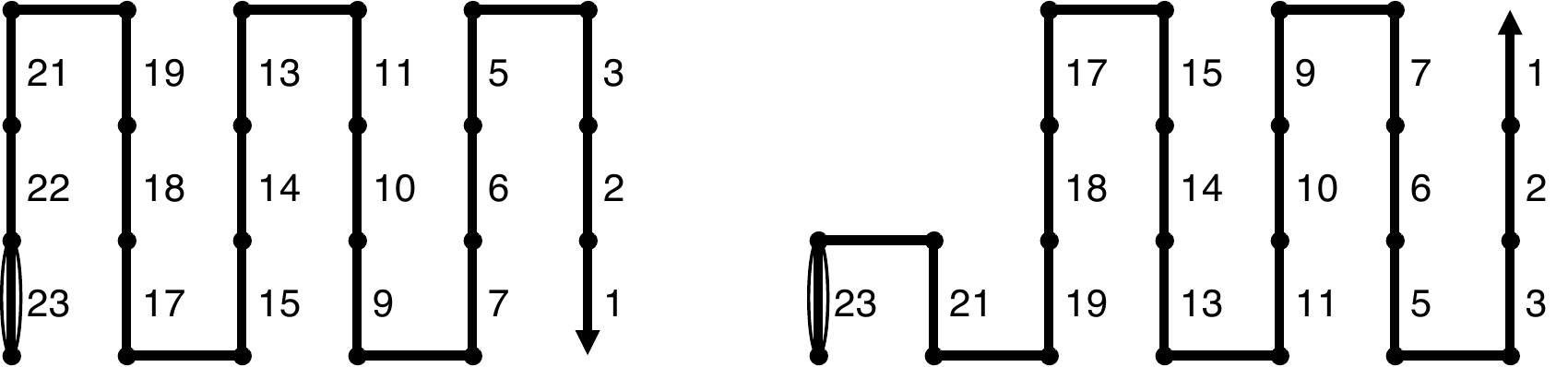}
	\caption{The left justified robot $L_{3,23}$ and the shifted left justified robot $L_{3,23}$.}
	\label{fig:robot_left}
\end{figure}

\begin{theorem}\label{thm:diameter}
The diameter of the transition graph $G(R_{m,n})$ of a robot on length $n$ in a tunnel of width $m\geq 2$ is 
\[
\diam G(R_{m,n}) =
\left\{ \begin{array}{lcl}
        d(\Left,\Hor) & \mbox{for} & n<6 \\ 
	d(\Left,\Leftt)  & \mbox{for} & n\geq 6.
                \end{array}\right.
\]
These distances are explicitly given by
\begin{align*}
d(\Left,\Hor) &= s_n - z_{m,n} \\
d(\Left,\Leftt) &= s_{n-1} - z_{m,n} + s_{n-2} - z_{m,m+n-3},
\end{align*}
where $s_n= {n+1 \choose 2}$ and $z_{m,n}=(r+1)k+{k\choose 2}(m+1)$ with $n=(m+1)k+r$, $0\leq r \leq m$.
\end{theorem}

This answers a question of the first author at the Open Problem Session of the Encuentro Colombiano de Combinatoria ECCO 2014 \cite{ArdilaECCO}.

\subsection{\textsf{Proof of Proposition~\ref{prop:distance}}}\label{sec:distance_proof}

Let $T$ and $T'$ be two coral tableaux of fixed type $(m,n)$, and let~$\lambda$ and~$\lambda'$ their corresponding shapes. 
We decompose $T$ into two parts $T_1$ and $T_2$ of shapes $\lambda\wedge \lambda'$ and~$\lambda \smallsetminus (\lambda\wedge \lambda')$ respectively. Similarly, we decompose $T'$ into two parts $T_1'$ and $T_2'$
of shapes $\lambda\wedge \lambda'$ and~$\lambda' \smallsetminus (\lambda\wedge \lambda')$ respectively.   
Additionally, we create another filling of $\lambda$ starting with the number $n$ in the bottom left square and decreasing each time by one along the snake. The restriction of this filling to the shape~$\lambda \smallsetminus (\lambda\wedge \lambda')$ is denoted by $\overline{T_2}$. The filling  $\overline{T_2'}$ of ~$\lambda' \smallsetminus (\lambda\wedge \lambda')$ is defined analogously.
These constructions are llustrated in Figure~\ref{fig:snake_decomposition}.  

\begin{figure}[htbp]
	\centering
		\begin{overpic}[width = \textwidth]{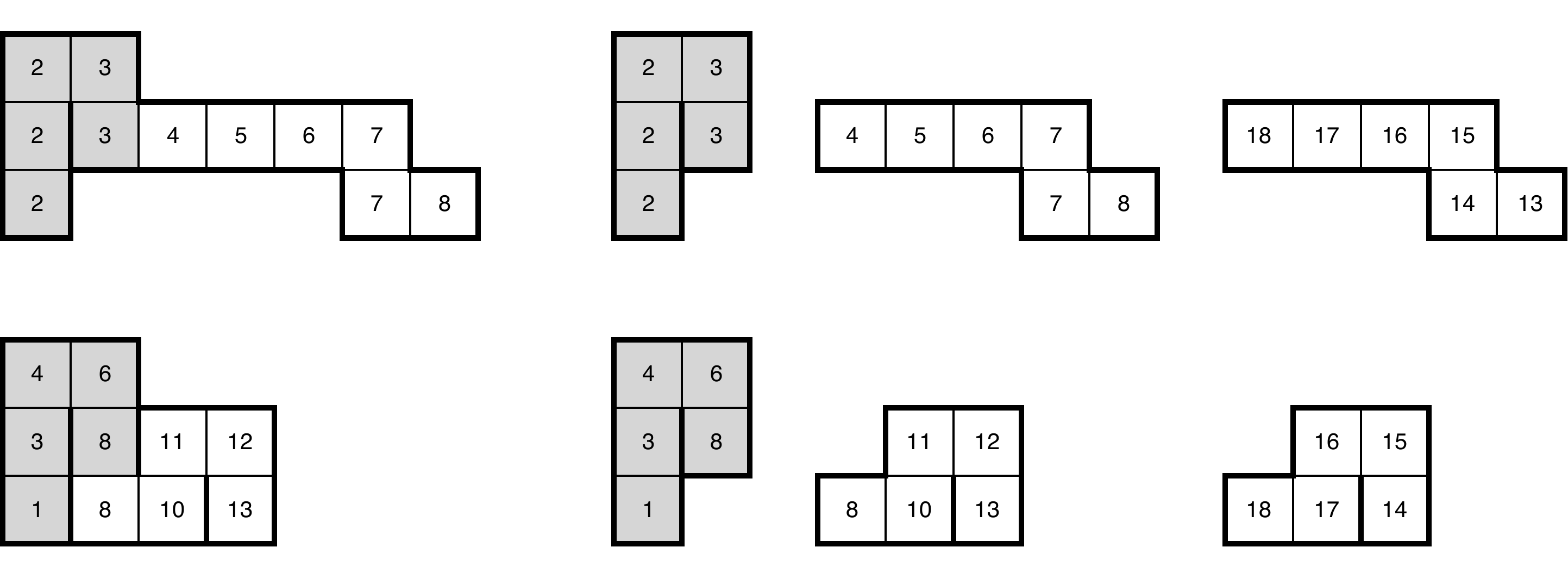}
			\put(8,18.5){\small $T$}
			\put(42,18.5){\small $T_1$}
			\put(58,18.5){\small $T_2$}
			\put(84,18.5){\small $\overline{T_2}$}

			\put(8,-1){\small $T'$}
			\put(42,-1){\small $T_1'$}
			\put(58,-1){\small $T_2'$}
			\put(84,-1){\small $\overline{T_2'}$}
		\end{overpic}
	\caption{Decomposition of two coral tableaux $T$ and $T'$ of type $(3,23)$.}
	\label{fig:snake_decomposition}
\end{figure}
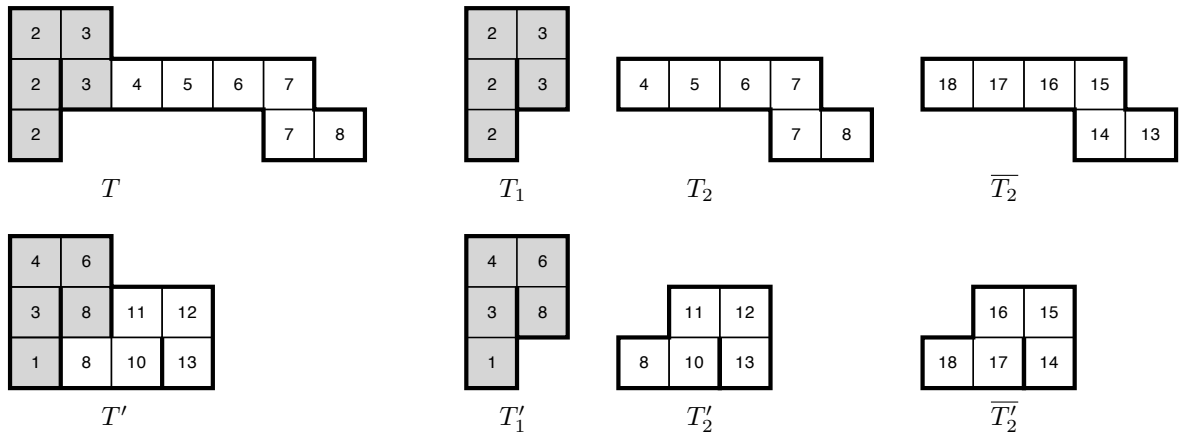

Given two fillings $T$ and $T'$ of the same shape $\lambda$, we denote by $T-T'$ the filling of $\lambda$ whose entries are the differences between the entries in $T$ and $T'$. Note that some of the entries of~$T-T'$ might be negative. Recall that $|T-T'|$ is the sum of the absolute values of $T-T'$.

\begin{lemma}
Let $P$ and $P'$ be two positions of the robotic arm $R_{m,n}$, and $T$ and $T'$ be the corresponding coral tableaux . Using the notation above, the distance between $P$ and $P'$ in the transition graph $G(R_{m,n})$ is equal to 
\[
d(P,P') = |\overline{T_2}-T_2| + |T_1-T_1'| + |\overline{T_2'}-T_2'|.
\]
\end{lemma}

In the example of Figure~\ref{fig:snake_decomposition}, $d(P,P')=56+12+26=94$. We will show a shortest path subdivided into three parts connecting $T\rightarrow T_1\rightarrow T_1'\rightarrow T'$ of lengths 56, 12, and 26.

\begin{proof}
First, we prove that the number of moves needed to move the robot $P$ to $P'$ is at least the claimed number. For this we analyze the possible moves of the robot in terms of the corresponding coral tableau. There are two kinds of moves: switching a corner of the robot corresponds to either increasing or decreasing an entry of the coral tableau by one, such that the result is still a coral tableau. Flipping the end of the robot corresponds to either deleting or adding a last box of the tableau when its entry has the maximum possible value $n-l(\lambda)$, in agreement with Definition \ref{def:coralsnaketableau}. For simplicity, we call these steps \emph{allowable tableau moves}.  
  
Let $\lambda$ and $\lambda'$ be the shapes of $T$ and $T'$ respectively. Let $T_1,T_2,\overline{T_2}$ and $T_1',T_2',\overline{T_2'}$ as above. 
In order to move from $T$ to $T'$ with allowable tableau moves it is necessary to: \\
(i) make disappear all entries of $T_2$ in $T$,\\
(ii) convert all the entries of $T_1$ to the entries of $T_1'$, and \\
(iii) make appear all entries in $T_2'$.
We claim that the number of moves required to do these three steps is at least $|\overline{T_2}-T_2|$, $|T_1-T_1'|$ and $|\overline{T_2'}-T_2'|$ respectively, from which we deduce that the distance between~$T$ and~$T'$ is at least 
\[|\overline{T_2}-T_2| + |T_1-T_1'| + |\overline{T_2'}-T_2'|.\] 

To prove the first claim, let $t$ and $\overline{t} = n-l(\lambda)+1$ be the last entries of $T_2$ and $\overline{T_2}$, respectively. To make the last cell of $T_2$ disappear, we need to perform $n - l(\lambda) - t$ allowable tableau moves at that entry to increase it to its maximum value of $n-l(\lambda)$, and one additional move to remove the cell, for a total of $|\overline{t}-t|$ moves. Continuing analogously, we see that if we wish to achieve (i) we need to perform at least $|\overline{T_2} - T_2|$ tableau moves in $\lambda' - (\lambda \wedge \lambda')$.
Similarly, to achieve (iii) we need at least $|\overline{T_2'}-T_2'|$ in $\lambda - (\lambda \wedge \lambda')$. Finally, since each allowable move changes an entry by one, we require at least $|T_1-T_1'|$ moves in 
$(\lambda \wedge \lambda')$ to carry out (ii).

%

The previous argument is also a roadmap for how to achieve this lower bound. 
We first go from tableau $T$ to $T_1$ in $|\overline{T_2}-T_2|$ moves by removing the boxes of $T_2$ in order from the last to the first. Then we then go from $T_1$ to $\operatorname{max}\{T_1,T_1'\}$ by increasing the entries one at a time from the last to the first box, and from $\operatorname{max}\{T_1,T_1'\}$ to $T_1'$ by decreasing the entries one at a time, again from the last to the first box; this takes a total of $|T_1-T_1'|$ moves.
Finally,  we connect $T_1'$ to $T'$ by adding all the entries in~$T_2'$ from first to last, using $|\overline{T_2'}-T_2'|$ moves.   
\end{proof}

Proposition~\ref{prop:distance} is now a direct consequence of the previous lemma.

\begin{proof}[Proof of Proposition~\ref{prop:distance}]
Let $P$ and $P'$ be two positions of the robotic arm $R_{m,n}$, and $T$ and~$T'$ be their corresponding coral tableaux. We need to show that
\[
|w|=|\overline{T_2}-T_2|, \quad \quad |v-v'|=|T_1-T_1'|, \quad  \quad  |w'|=|\overline{T_2'}-T_2'|.
\]  
The entry of a box in~$\overline{T_2}-T_2$ counts the number of links in $P$ after the corresponding vertical link, including it. Therefore, the entries of $w$ are exactly equal to the entries of $\overline{T_2}-T_2$. Similarly, the entries of $w'$ are equal to the entries of $\overline{T_2'}-T_2'$. 
Now, if $v=(v_1,\dots,v_k)$ and if $(t_1,\dots, t_k)$ are the entries of $T_1$ in the order they appear along the snake, then $v_i=n-(t_i+i-1)$ for each $i$. Therefore the entries of $v-v'$ are exactly the negatives of the entries of $T_1-T_1'$, and hence $|v-v'|=|T_1-T_1'|$.   
\end{proof}

\subsection{\textsf{Proof of Theorem~\ref{thm:diameter}}}\label{sec:diameter_proof}

Throughout this section we fix $m\geq 2$ and $n\in \mathbb N$. 
We say that a coral snake $\lambda$ is of type~$(m,n)$ if it appears in the coral PIP $C_{m,n}$; that is, if $h(\lambda)\leq m$ and $l(\lambda)+w(\lambda)-1\leq n$.  
All the coral snakes and positions of the robot considered in this section are of type $(m,n)$.

\begin{definition}
The \emph{shape} of a position $P$ of the robot, denoted $\shape(P)$,  is the shape of its corresponding coral tableau. The \emph{intersection shape} of two positions $P$ and $P'$ is defined as~$\shape(P)\wedge \shape(P')$. 
\end{definition}

Our strategy to find the diameter of the transition graph $G(R_{m,n})$ is to find the maximum distance between two positions of the robot with a fixed intersection shape~$\lambda$, and then maximize this quantity over all  shapes $\lambda$. We need to distinguish two kinds of snakes.

\begin{definition}
A coral snake of type $(m,n)$ is said to be a:\\
$\bullet$
side snake: if the end point of any robot with coral tableau of shape $\lambda$ is on one of the two horizontal sides of the tunnel. \\
$\bullet$
middle snake: if it is not a side snake.  
\end{definition}  

\begin{definition}
We define the following special positions of the robot (see Figure~\ref{fig:robots_specialpositions}):\\
$\bullet$ $\LeftLambda$: the position with componentwise maximum vertical labelling among all positions whose coral tableau contains the shape $\lambda$.\\
$\bullet$  $\LefttLambda$: the position with componentwise maximum vertical labelling among all positions whose coral tableau contains the shape $\lambda$ and whose intersection shape with $\LeftLambda$ is $\lambda$. \\
$\bullet$ $\HLambda$: the position with componentwise minimum vertical labelling among all robots whose coral tableau has shape $\lambda$. 
\end{definition}

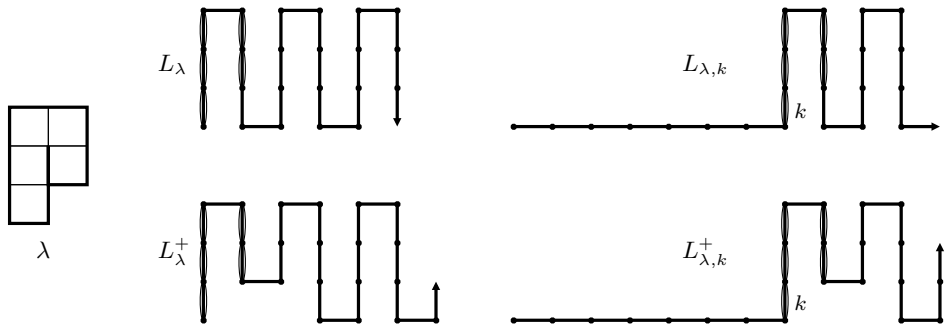
\begin{figure}[htbp]
	\centering
	\begin{overpic}[width=0.8\textwidth]{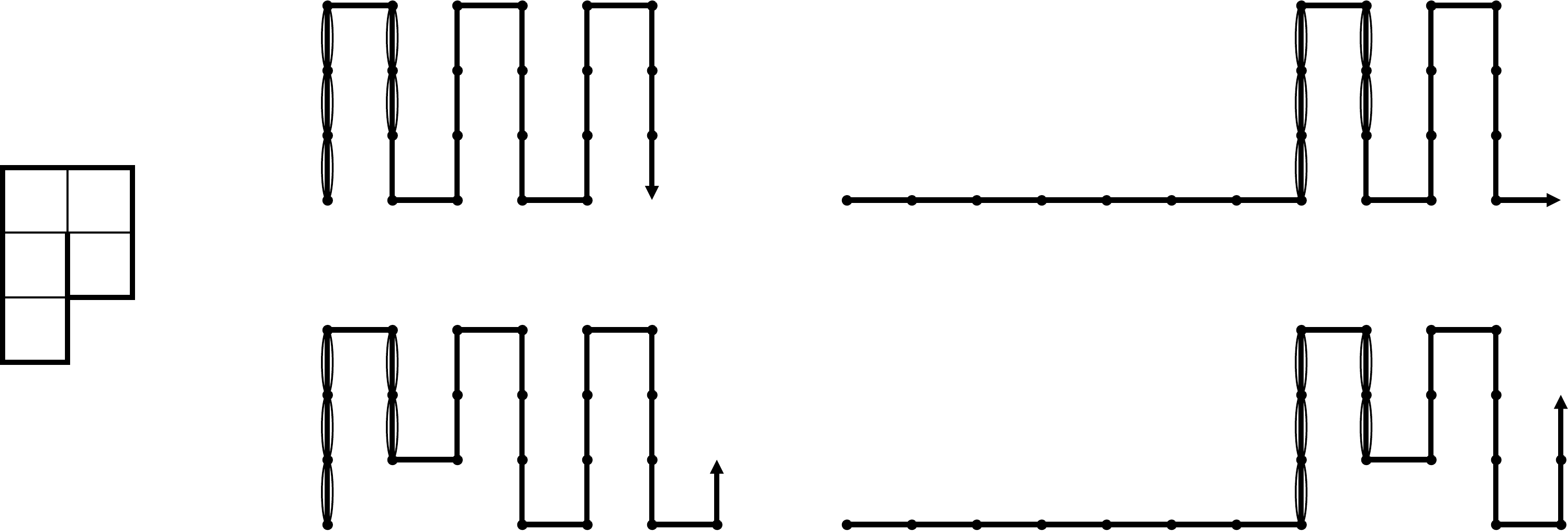}
		\put(3,7){\footnotesize $\lambda$}
		\put(16,27){\footnotesize $\LeftLambda$}
		\put(16,7){\footnotesize $\LefttLambda$}
		\put(72,27){\footnotesize $\LeftLambdak$}
		\put(72,7){\footnotesize $\LefttLambdak$}
		\put(84,1.5){\scriptsize $k$}
		\put(84,22){\scriptsize $k$}
	\end{overpic}
	\caption{Four special positions of the robot for the given $\lambda$, $(m,n)=(3,23)$ and $k=16$.}
	\label{fig:robots_specialpositions}
\end{figure}

The following two technical lemmas are the key steps to proving Theorem \ref{thm:diameter}.

\newpage

\begin{lemma}\label{lem:d1}
The maximum distance between two positions of the robot with fixed intersection shape $\lambda$ is:
\begin{enumerate}[(i)]
\item $d(\LeftLambda,\HLambda)$, if $\lambda$ is a side snake.
\item $\max\{ d(\LeftLambda,\HLambda), d(\LeftLambda,\LefttLambda) \}$, if $\lambda$ is a middle snake. 
\end{enumerate}
\end{lemma}

\begin{lemma}\label{lem:d2}
The following hold:
\begin{enumerate}[(i)]
\item If $\lambda \prec \lambda'$, then 
$d(\LeftLambda,\HLambda)>d(\LeftLambdaP{\lambda'},\HLambdaP{\lambda'})$.
\item If $\lambda \prec \lambda'$ are two middle snakes, then 
$d(\LeftLambda,\LefttLambda)>d(\LeftLambdaP{\lambda'},\LefttLambdaP{\lambda'})$.
\end{enumerate}

\end{lemma}

We postpone the proof of these two lemmas for the moment and use them to prove our main diameter result. 

\begin{proof}[Proof of Theorem~\ref{thm:diameter}]
By Lemma~\ref{lem:d1}, the diameter of the transition graph $G(R_{m,n})$ is the maximum between the two values 
\[
\max_{\lambda} d(\LeftLambda,\HLambda), \quad   \quad  \max_{\text{middle snakes }\lambda} d(\LeftLambda,\LefttLambda).
\] 
By Lemma~\ref{lem:d2}, we have
\[
\max_{\lambda} d(\LeftLambda,\HLambda)=
d(\LeftLambdaP{\emptyset},\HLambdaP{\emptyset}), \quad   \quad  
\max_{\text{middle snakes }\lambda} d(\LeftLambda,\LefttLambda)= 
d(\LeftLambdaP{\square},\LefttLambdaP{\square}).
\] 
where $\emptyset$ is the empty snake and $\square$ is the snake that consists of only one box.
By definition, we have that $\LeftLambdaP{\emptyset}=\LeftLambdaP{\square}=\Left$, 
$\HLambdaP{\emptyset} = \Hor$ and $\LefttLambdaP{\square}=\Leftt$.
Therefore,
\[
\diam G(R_{m,n}) = \max \{d(\Left, \Hor), d(\Left, \Leftt)\}.
\]

The explicit formulas for these distances stated in Theorem \ref{thm:diameter} are obtained directly from Proposition~\ref{prop:distance}. For instance, $d(\Left, \Hor)$ is the sum of all vertical labels in $\Left$. This sum is equal to the sum $s_n=1+\dots+n$ of all labels (including the horizontal ones) minus the sum $z_{m,n}$ of the horizontal labels. The formula for $d(\Left, \Leftt)$ is obtained similarly, considering the labels of $\Leftt$ as well.   

When $n \geq 4$, we may rewrite $d(\Left,\Leftt) = d(\Left, \Hor) + s_{n-4} - z_{m,n-4} - 2$, which shows that $d(\Left,\Leftt) \geq d(\Left, \Hor)$ for $n \geq 6$ and $d(\Left,\Leftt) < d(\Left, \Hor)$ for $n=4, 5$. The cases $n=1,2,3$ are easily verified by hand.
%
\end{proof}

\subsubsection{\textsf{Fixing a shape: proof of Lemma~\ref{lem:d1}}}

Let us prove Lemma \ref{lem:d1} for any fixed coral snake $\lambda$ of type $(m,n)$. Let $P$ and $P'$ be two positions of the robot with intersection shape $\lambda$. We begin by proving case (ii), which is more intricate, and then return to case (i).

\begin{proof}[Proof of Lemma~\ref{lem:d1}~$(ii)$]
Let $\lambda$ be a middle snake. 
 For $\ell(\lambda)+w(\lambda)-1 \leq k\leq n$ the following special positions of the robot (see Figure~\ref{fig:robots_specialpositions}) will play an important role in the proof:
  
 \begin{itemize}
\item $\LeftLambdak$: the position with componentwise maximum vertical labelling among all positions whose entries are less than or equal to $k$ and whose coral tableau contains the shape $\lambda$
.
\item  $\LefttLambdak$: the position with componentwise maximum vertical labelling among all positions whose entries are less than or equal to $k$,  whose coral tableau contains the shape $\lambda$, and whose intersection shape with $\LeftLambdak$ is $\lambda$.
\end{itemize}

Notice that $\LeftLambdak$ (resp. $\LeftLambdak^+$) consists of $n-k$ horizontal steps followed by the position analogous to $\LeftLambda$ (resp. $\LeftLambda^+$) for the robotic arm of length $k$.
In particular, $\LeftLambda = \LeftLambdaP{\lambda,n}$ and $\LefttLambda = \LefttLambdaP{\lambda,n}$. 
The description of the maximum distance between two robots with fixed intersection shape $\lambda$ will follow from the following three lemmas.


%

\begin{lemma}\label{lem:maximizingdistancewithgivenshape2}
For any positions $P$ and $P'$ with intersection shape $\lambda$ there exists $k$ such that 
 $d(P,P')\leq \max\{d(\LeftLambda,\LefttLambdak),  d(\LefttLambda,\LeftLambdak)\}$.
\end{lemma}

\begin{proof} Recall that
\begin{equation*}
d(P,P') = |w| + |v-v'| + |w'|.
\end{equation*}
We will perform two distance-increasing  transformations to the pair $(P,P')$ to obtain either
$(\LeftLambda,\LefttLambdak)$ or 
$(\LefttLambda,\LeftLambdak)$.
Without loss of generality assume that, among the vertical labels in $v$ and $v'$, the minimum $m$ is achieved in $v'$. 

The first transformation moves the vertical steps corresponding to $v$ in $P$ as far to the left as possible (by increasing their values in $v$), and moves the vertical steps corresponding to $v'$ in $P'$ as far to the right as possible (by decreasing their values in $v'$) while keeping the minimal label equal to~$m$. Denote by $P_1$ and $P_1'$ the resulting positions;  Figure~\ref{fig:robots_maximizing_distance} shows an example. This transformation preserves the values of $|w|$ and $|w'|$ in the formula and increases the value of $|v-v'|$. 
 
The second transformation changes the parts of $P_1$ and $P_1'$ after their intersection $\lambda$, in order to maximize $|w|$ and $|w'|$ while keeping $|v-v'|$ fixed. 
We need to keep the intersection shape equal to $\lambda$, so we must respect the up/down direction of the next vertical step in $P_1$ and $P_1'$ after $\lambda$; we convert the rest of $P_1$ and $P_1'$ into  ``zigzag snakes" subject to that constraint. The result is either the pair $(\LeftLambda,\LefttLambdak)$ or $(\LefttLambda,\LeftLambdak)$ for some $k$. 
It follows that 
\[
d(P,P') \leq d(P_1,P_1') \leq \max\left\{ d(\LeftLambda,\LefttLambdak), d(\LefttLambda,\LeftLambdak) \right\},
\]
as desired. This construction is illustrated in Figure \ref{fig:robots_maximizing_distance}.
\end{proof}

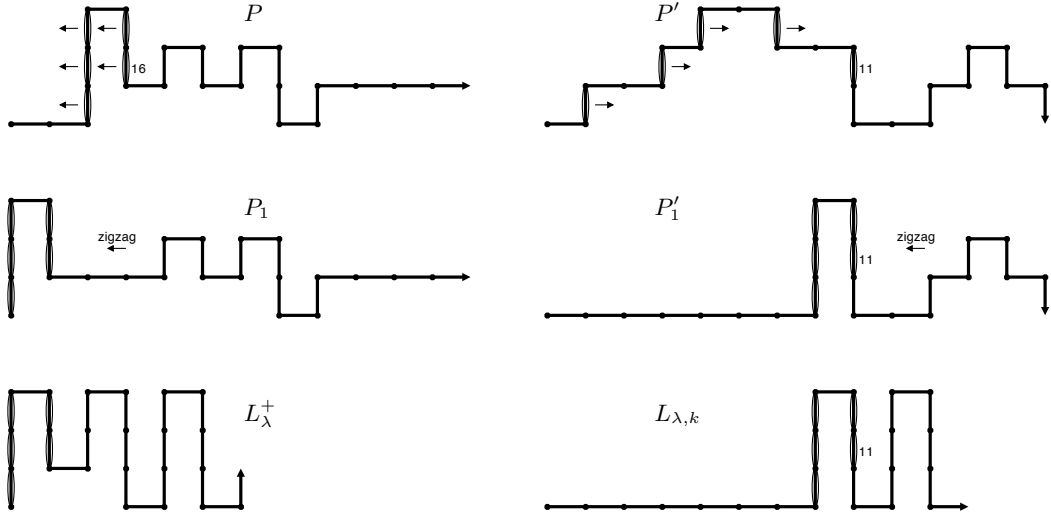
\begin{figure}[htbp]
	\centering
	\begin{overpic}[width=0.9\textwidth]{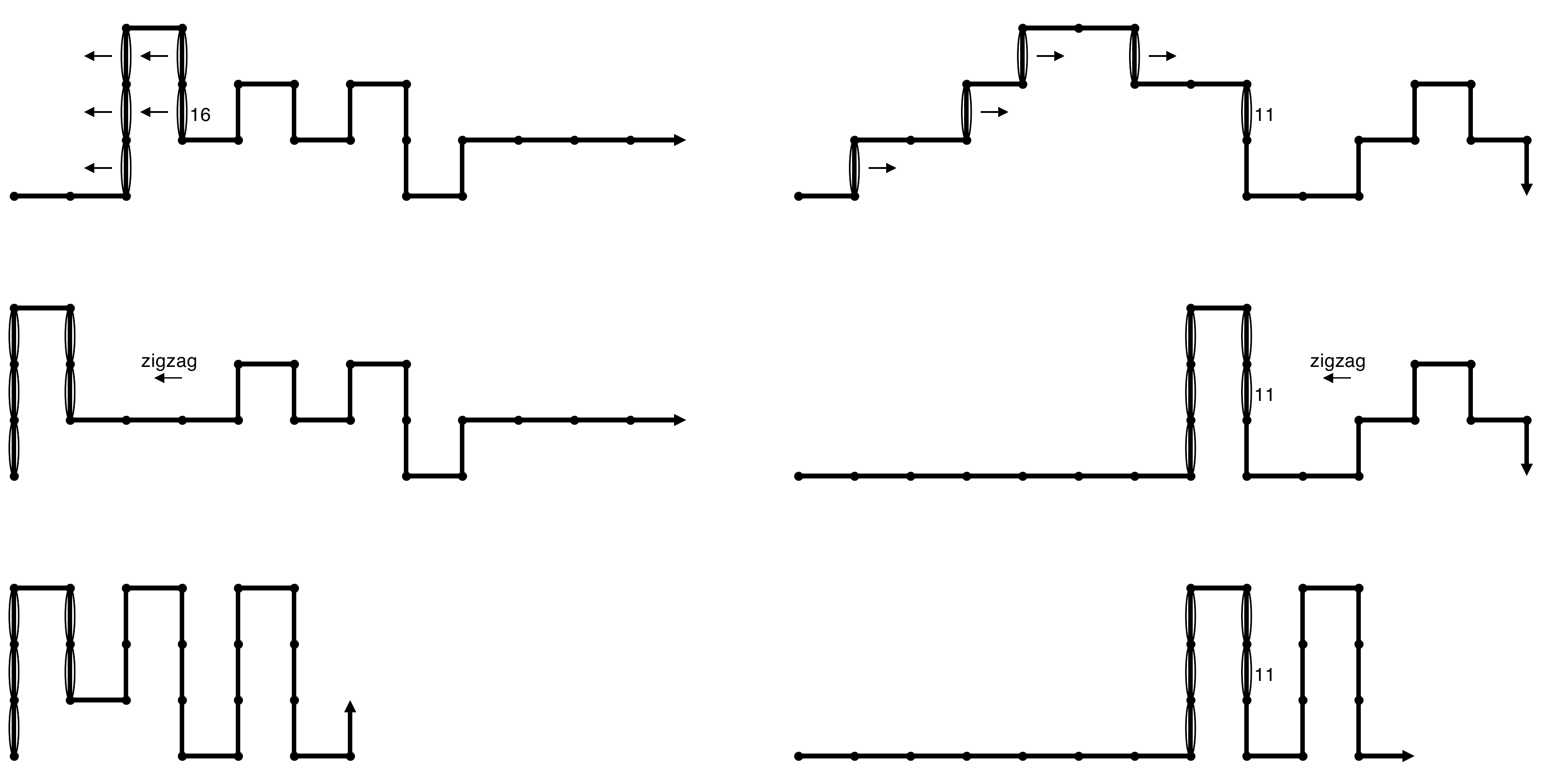}
		\put(23,48){\footnotesize $P$}
		\put(23,29.5){\footnotesize $P_1$}
		\put(23,10){\footnotesize $\LefttLambda$}
		\put(62,48){\footnotesize $P'$}
		\put(62,29.5){\footnotesize $P_1'$}
		\put(62,10){\footnotesize $\LeftLambdak$}
	\end{overpic}
	\caption{Increasing the distance between two robots with a given intersection shape.}
	\label{fig:robots_maximizing_distance}
\end{figure}

\begin{lemma}\label{lem:maximizingdistancewithgivenshape1}
$d(\LeftLambda,\LefttLambdak) \geq d(\LefttLambda,\LeftLambdak)$.
\end{lemma}

\begin{proof}
For simplicity, let $\ell=n-k$ be the number of initial horizontal steps of $\LeftLambdak$ and~$\LefttLambdak$ (if $\lambda$ is non-empty). 
As above, label the links of the robot in decreasing order from~$n$ up to~$1$ along the shape of the robot.
Let $h,h^+,h_k,h_k^+$ be the restriction of this labelling to the horizontal links located after the last vertical link corresponding to $\lambda$ in each of the robots $\LeftLambda$, $\LefttLambda$, $\LeftLambdak$, $\LefttLambdak$ respectively. 

The vector $h - h_k$ (resp. $h^+ - h_k^+$) consists of an initial sequence of $(n-k)$s followed by the entries of $h$ (resp. $h^+$) that are less than $n-k$. 
Consider the injection taking the $i$th horizontal step of $L_\lambda$ to the $i$th horizontal step of $L_\lambda^+$, which has a weakly larger label.
Under this injection, each term of $h-h_k$ is dominated by the corresponding term in $h^+ - h_k^+$. It follows that $(|h^+-h_k^+|) - (|h-h_k|) \geq 0$.
%
\end{proof}
%
%

\begin{lemma}\label{lem:maximizingdistancewithgivenshape3}
For a fixed middle snake $\lambda$, $d(\LeftLambda,\LefttLambdak)$ is  concave up as a function of $k$, and therefore attains its maximum value either at $k=l(\lambda) + w(\lambda)-1$ or at $k=n$.
\end{lemma}

\begin{proof}
It suffices to show that the (possibly negative) difference $d(\LeftLambda,\LefttLambdaP{\lambda,k+1})-d(\LeftLambda,\LefttLambdak)$ is increasing for every $k$ where it is defined. Let $v^+_k\circ w^+_k$ be the $(\LeftLambda,\LefttLambdak)$-decomposition of $\LefttLambdak$, and $\ell(w^+_k)$ be the number of entries in $w^+_k$. If $\ell(\lambda)$ is the number of boxes of the intersection shape, then 
\[
d(\LeftLambda,\LefttLambdaP{\lambda,k+1})-d(\LeftLambda,\LefttLambdak) = 
\ell(\lambda) - \ell(w^+_k).
\]
The result follows from the fact that $\ell(w^+_k)$ is decreasing in $k$ and $\ell(\lambda)$ is constant.
\end{proof}

We finally have all the ingredients to complete the proof of Lemma~\ref{lem:d1}(ii). 
Let $\lambda$ be a middle snake. 
By Lemmas~\ref{lem:maximizingdistancewithgivenshape2} and ~\ref{lem:maximizingdistancewithgivenshape1},
\[
d(P,P') \leq d(\LeftLambda,\LefttLambdak)
\] 
for some value $k$.
By Lemma~\ref{lem:maximizingdistancewithgivenshape3}, the maximum of $d(\LeftLambda,\LefttLambdak)$ is achieved either at $k=n$ or $k=\ell(\lambda)+w(\lambda)-1$. These two cases correspond precisely to $d(\LeftLambda,\LefttLambda)$ and $d(\LeftLambda,\HLambda)$ respectively. Therefore, 
\[
d(P,P') \leq \max\{ d(\LeftLambda,\HLambda), d(\LeftLambda,\LefttLambda) \}
\]
as we wished to show. 
\end{proof}

Now we go back to case (i), which is easier.

\begin{proof}[Proof of Lemma~\ref{lem:d1}~$(i)$]
If $\lambda$ is a side snake, then either $\shape(P)=\lambda$ or $\shape(P')=\lambda$, otherwise their intersection shape would be bigger than $\lambda$. 
For simplicity assume ${\shape(P')=\lambda}$. 
Again, we transform $P$ and $P'$ to increase their distance: 
we move the vertical steps corresponding to $\lambda$ in $P$ as far to the left as possible, and the vertical steps of $P'$ as far to the right as possible. This transformation preserves the values of $|w|$ and $|w'|$ in the formula and increases the value of $|v-v'|$. In addition, transform the part of the modified $P$ after $\lambda$ in order to maximize $|w|$. The result is the pair $(\LeftLambda,\HLambda)$.
\end{proof}

\subsubsection{\textsf{Varying the shape: proof of Lemma~\ref{lem:d2}}}

\begin{proof}[Proof of Lemma~\ref{lem:d2}~$(i)$]
By Proposition~\ref{prop:distance}, removing the last box of $\lambda$ increases the distance $d(\LeftLambda,\HLambda)$ by $\ell(\lambda)$. Therefore $\lambda \prec \lambda'$ implies $d(\LeftLambda,\HLambda) > d(L_{\lambda'},H_{\lambda'})$.
\end{proof}

To prove part (ii) we 
use a basic lemma.

\begin{lemma}\label{lem:numberswithgaps}
Let $A$ be a subsequence of $(p,\dots, 2,1)$ obtained by removing an arithmetic progression with common difference $c$, and let $A'$ be a  subsequence of $(p',\dots, 2,1)$ obtained by removing an arithmetic progression with common difference $c$. If $p>p'$ then $|A| \geq|A'|$. 
\end{lemma}

\begin{proof}
Say $A$ and $A'$ have $a$ and $a'$ elements, respectively. 
When listed in decreasing order, the largest $a'$ terms of $A$ are greater than or equal to the $a'$ terms of $A'$. The remaining $a-a'$ terms (if there are any) are positive. Therefore $|A| \geq |A'|$.
\end{proof}

\begin{proof}[Proof of Lemma~\ref{lem:d2}~$(ii)$]
Assume that $\lambda \prec \lambda'$ are two middle snakes. We need to show that $d(\LeftLambda,\LefttLambda) = |w| + |v-v'| + |w'|$ decreases as we increase $\lambda$. 
The term 
$|v-v'|=0$ 
stays constant.
The tem $|w|$ clearly does not decrease when we make $\lambda$ larger. Finally, the term $|w'|$ also does not decrease when we make $\lambda$ thanks to Lemma~\ref{lem:numberswithgaps}. The desired result follows.
\end{proof}

\section{\textsf{Implementation of the shortest path algorithm}\label{sec:algorithm}}

Since the configuration space of the robotic arm $R_{m,n}$ is CAT(0), we are able to navigate it thanks to the following theorem. 

\begin{reptheorem}{th:alg} \cite{AG, ABY, Re} 
If the configuration space of a robot is a $\mathrm{CAT}(0)$ cubical complex, there is an explicit algorithm to move the robot optimally from one position to another, in:

\noindent $\bullet$ \emph{the edge metric}: the number of moves, if only one move at a time is allowed,

\noindent $\bullet$ \emph{the cube metric}: the number of steps (where in each step we may perform several physically independent moves),

\noindent $\bullet$ \emph{the time metric}: the time elapsed if we may perform independent moves simultaneously.

\end{reptheorem}

As shown in \cite{AOS}, there is also a (more complicated) algorithm to move the robot optimally in terms of the Euclidean metric in the configuration space. However, this metric is less relevant, as it does not seem to have a natural interpretation in terms of the robot.

The algorithms of Theorem \ref{th:alg} are described in detail in~\cite{ABY}; they are based on the \emph{normal cube paths} of Niblo and Reeves \cite{NR}. 
We have implemented these algorithms in Python for the robotic arms discussed in this paper. 
To use the program, the user inputs the width of the tunnel and two positions of the robotic arm of the same length, expressed as sequences of steps of the form: r (right), u (up), d (down) steps.
The program outputs the distance between the two states in the edge metric and in the cube metric\footnote{It turns out that the optimal paths in the cube metric are also optimal in the time metric.}, and an animation taking the robot between the two states. 
The downloadable code, instructions, and a sample animation are at \href{http://math.sfsu.edu/federico/Articles/movingrobots.html}{\texttt{http://math.sfsu.edu/federico/Articles/movingrobots.html}}.
%

\section{\textsf{Future directions and open problems}}

\begin{itemize}

\item
Our robotic arms $R_{m,n}$ can never have links facing left. 
Now consider the robotic arm of length $n$ in a tunnel of width $m$ whose links may face up, down, right, \textbf{or left}, starting in the fully horizontal position, and subject to the same kinds of moves: switching corners and flipping the end. Is the configuration space of this robot also CAT(0)? In \cite{ABCG1}, we show the answer is affirmative for width $m=2$. We believe this is also true for larger $m$, but proving it in general seems to require new ideas.

\item
Theorem \ref{thm:diameter} gives the diameter of the configuration space $\S_{m,n}$ in the edge metric, that is, the largest number of moves separating two positions of the robot $R_{m,n}$.
What is the diameter of $\S_{m,n}$ in the cube metric, that is, the largest number of steps separating two positions of the robot $R_{m,n}$, where a step may consist of several physically independent moves?

\end{itemize}

\section{\textsf{Acknowledgements}}
\label{sec:ack}
This paper includes results from HB's undergraduate thesis at Universidad del Valle in Cali, Colombia (advised by FA and CC) and JG's undergraduate project at San Francisco State University  in San Francisco, California (advised by FA). We thank the SFSU-Colombia Combinatorics Initiative, which made this collaboration possible. We would also like to thank the Pacific Ocean for bringing us the beautiful coral that inspired Definition \ref{def:snakePIP} and the proof of our main Theorem~\ref{conj:PIP}.

\bibliographystyle{amsalpha}

\end{document}